\documentclass[reqno,11pt]{amsart}
\numberwithin{equation}{section}

\usepackage[utf8]{inputenc}
\usepackage[italian,english]{babel}
\usepackage{amsmath}
\usepackage{amsfonts}
\usepackage{amssymb,esint}
\usepackage{tikz}
\usepackage{geometry}
\usepackage{color}
\usepackage{mathrsfs}
\usepackage{mathtools}

\mathtoolsset{showonlyrefs}

\usepackage[colorlinks, citecolor=citegreen, linkcolor=refred, urlcolor=blue]{hyperref}
\definecolor{citegreen}{rgb}{0,0.4,0}
\definecolor{refred}{rgb}{0.5,0,0}

{\catcode `\@=11 \global\let\AddToReset=\@addtoreset}
\AddToReset{equation}{section}
\renewcommand{\theequation}{\thesection.\arabic{equation}}
\newcounter{mnotecount}[section]

\renewcommand{\themnotecount}{\thesection.\arabic{mnotecount}}

\newcommand{\mnote}[1]
{\protect{\stepcounter{mnotecount}}$^{\mbox{\footnotesize
$
\bullet$\themnotecount}}$ \marginpar{
\raggedright\tiny\em
$\!\!\!\!\!\!\,\bullet$\themnotecount: #1} }

\newcommand{\jj}[1]%
{{\color{red}\mnote{{\color{red}{\bf jj:} #1} }}}

\theoremstyle{plain}
\newtheorem {theorem}{Theorem}[section]
\newtheorem {lemma}[theorem]{Lemma}
\newtheorem {proposition} [theorem]{Proposition}
\newtheorem {corollary} [theorem]{Corollary}

\newtheorem{definition}[theorem]{Definition}
\theoremstyle{remark}
\newtheorem{remark}[theorem]{Remark}
\newtheorem{example}[theorem]{Example}

\DeclarePairedDelimiter\abs{\lvert}{\rvert}

\renewcommand{\theequation}{\arabic{section}.\arabic{equation}}
\renewcommand{\thetheorem}{\arabic{section}.\arabic{theorem}}

\newcommand{\R}{\mathbb R}
\newcommand{\N}{\mathbb N}

\renewcommand{\theta}{\vartheta}

\newcommand{\barint}
{\rule[.036in]{.12in}{.009in}\kern-.16in \displaystyle\int}

\newcommand{\dive}{{\mathrm{div}}}

\newcommand{\pa}{\partial }

\usepackage{color}

\newcommand{\numberset}{\mathbb}
\renewcommand{\N}{\numberset{N}}
\renewcommand{\R}{\numberset{R}}
\newcommand{\Sf}{\numberset{S}}

\newcommand{\AVR}{{\rm AVR}(g)}
\newcommand{\capa}{{\rm Cap}}
\newcommand{\D}{{\rm D}}
\newcommand{\dd}{{\,\rm d}}
\newcommand{\HH}{{\rm H}}
\newcommand{\hh}{{\rm h}}

\newcommand{\Om}{\Omega}
\renewcommand{\phi}{\varphi}
\renewcommand{\epsilon}{\varepsilon}

\title[Minimising hulls, p-capacity and isoperimetric inequality]{Minimising hulls, p-capacity and isoperimetric inequality on complete Riemannian manifolds}

\author[M.~Fogagnolo]{Mattia Fogagnolo}
\address{M.~Fogagnolo, Centro di Ricerca Matematica Ennio De Giorgi, Scuola Normale Superiore,
Piazza dei Cavalieri 3, 56126 Pisa (PI), Italy}
\email{mattia.fogagnolo@sns.it}

\author[L.~Mazzieri]{Lorenzo Mazzieri}
\address{L.~Mazzieri, Universit\`a degli Studi di Trento,
via Sommarive 14, 38123 Povo (TN), Italy}
\email{lorenzo.mazzieri@unitn.it}

\begin{document}
\begin{abstract}
{
The notion of strictly outward minimising hull is investigated for open sets of finite perimeter sitting inside a complete noncompact Riemannian manifold. Under natural geometric assumptions on the ambient manifold, the strictly outward minimising hull $\Om^*$ of a set $\Om$ is characterised as a maximal volume solution of the least area problem with obstacle, where the obstacle is the set itself. {In the case where $\Om$ has $\mathscr{C}^{1, \alpha}$-boundary, the area of $\pa \Om^*$ is recovered as the limit of the $p$-capacities of $\Om$, as $p \to 1^+$.
Finally, building on the existence of strictly outward minimising exhaustions, a sharp isoperimetric inequality is deduced
on complete noncompact manifolds with nonnegative Ricci curvature, provided $3 \leq n \leq 7$. }}

\end{abstract}


\maketitle



\section{Introduction and statement of the main results}
{
Given a bounded set $\Omega \subset \R^n$ with finite perimeter, the problem of finding and describing its best envelope is a classical problem in Calculus of Variations, whose first rigorous settlement should be found in~\cite{miranda-ostacoli}. We refer the reader to that article and to the references therein also for a nice historical account on the subject. 

At a first glance, the best envelope of $\Omega \subset \R^n$ can be conceived as a bounded set that minimises the perimeter among all of those sets that contain $\Omega$. However, easy examples show that this property is not sufficient to guarantee the uniqueness of the envelope. In other words, there are sets that can be enclosed into several area minimising envelopes. The question arises of which further requirement should be imposed in order to find the most meaningful mathematical object. A possible choice is the one proposed in~\cite{bassanezi-tamanini}, where the authors selected, among the area minimising envelopes of a given set, the one that minimises the Hausdorff distance to the set itself, obtaining the so called {\em minimal hull}. 

More recently, the fundamental settlement of the theory of weak solutions to the Inverse Mean Curvature Flow (IMCF for short) by Huisken and Ilmanen~\cite{Hui_Ilm} has revealed the importance of
selecting the best area minimising envelope with a substantially different criterion, highlighting at the same time the relevance of this kind of questions even when they are posed in a more general context, such as on some relevant classes of complete Riemannian manifolds. More concretely, the authors showed that if a set $E$ appears along the evolution designed by the weak IMCF at some positive time, then it is necessarily {\em strictly outward minimising}, i.e.,
its perimeter satisfies $P(E) \leq P(F)$ for any $F \supseteq E$, and the equality is achieved only when $E$ coincides with $F$ almost everywhere.
A consequence of this fact is that, no matter what the initial set $\Omega$ is, the weak IMCF istantaneously takes it to a new bounded set $\Omega^*$ that is {\em strictly outward minimising}. 

In $\R^n$ the set $\Om^*$ satisfies two further remarkable properties. First, among the area minimising envelopes of $\Om$, it is the one that maximises the volume. Secondly, among all bounded strictly outward minimising envelopes of $\Omega$, it is the one that minimises the volume. In this sense, it might be thought as the smallest strictly outward minimising envelope of $\Om$. In view of this latter observation we will refer to the set $\Omega^*$ as to the~\emph{strictly outward minimising hull} of $\Omega$. 

To summarise,
we have that in $\R^n$ the {\em strictly outward minimising hull} $\Om^*$ of a given bounded set $\Om$ with sufficiently smooth boundary can be equivalently characterised as: 
{\em
\begin{itemize}
\item[$(A)$] The minimal volume strictly outward minimising envelope of $\Om$;
\medskip
\item[$(B)$] The maximal volume solution to the least area problem with obstacle $\Om$;
\medskip
\item[$(C)$] The (measure theoretic) interior of the sublevel set $\{ w \leq 0\}$, where $w$ is the arrival time function in the weak formulation of the IMCF starting at $\pa \Om$.
\smallskip
\end{itemize}
}
Before proceeding, we point out that, albeit this trifold characterisation of the strictly outward minimising hulls seems to be partially known to experts, at least in the usual Euclidean setting, we were unable to find a systematic and fully detailed treatment of this subject in the literature, especially for what concerns the variational characterisation described in statement $(B)$. Part of our motivation for the present work comes precisely from the attempt of giving a consistent and uniform presentation of the whole picture, encompassing all the different  aspects of such a central notion.

On this regard, there is another fact about $\Om^*$ that  is worth mentioning, namely that its perimeter in $\R^n$ can be computed as the limit of the $p$-capacities of $\Om$ as $p \to 1^+$. Apart from its concrete application in the proof of the extended Minkowski Inequality~\cite[Theorem 1.1]{Ago_Fog_Maz_2}, we found this result particularly interesting, since it reveals a profound relationship between the best envelope problem, the Huisken-Ilmanen's notion of weak IMCF and the Nonlinear Potential
Theory. As we are going to see, such a deep connection continues to be true even beyond the familiar Euclidean framework, somehow suggesting that a geometric evolution through weak IMCF should be well defined if and only if the perimeter of any strictly outward minimising hull is well approximated by the $p$-capacities of the corresponding obstacle.

\bigskip

\emph{Aim of this work is to discuss the rigorous definition and the variational properties of the strictly outward minimising hull of a given set of finite perimeter in the context of complete noncompact Riemannian manifolds, highlighting the relations with the notion of weak IMCF as well as with the geometric features of the nonlinear potential theory.}

\bigskip

We start observing that, no matter in which framework we are, any natural notion of strictly outward minimising hull is supposed to select a set that minimises the perimeter, among the strictly outward minimising envelopes of a given obstacle.
Let us mention that such a property was actually needed in the proof of the Isoperimetric Inequality \emph{with sharp constant} performed in~\cite{Ago_Fog_Maz_1} in the context of complete noncompact $3$-manifolds with nonnegative Ricci curvature, as well as in the one proposed by Schulze in~\cite{schulze1} on $3$-dimensional Cartan-Hadamard manifolds.

Having the Euclidean picture in mind, it might seem natural at a first glance to consider the class of bounded sets with finite perimeter and look for a good definition of strictly outward minimising hull, using either condition $(A)$ or $(B)$ as a paradigm.
However, easy examples show that, {in a general Riemannian ambient, the strictly outward minimizing hull $\Omega^*$, {roughly understood as the smallest bounded strictly outward minimising set containing $\Omega$}, might in general fail to be well defined.} Indeed, manifolds with cuspidal or cylindrical ends (see the discussion in Examples~\ref{example-cuspidal} and~\ref{example-cylindrical}, for more details)
provide natural examples of spaces in which one can find bounded obstacles that are not contained in any strictly outward minimising set, making impossible the construction via intersection of the corresponding strictly outward minimising hull. 

Several other pathologies are present in this class of manifolds. For example, in  manifolds with cuspidal ends, the least area problem with obstacle might have no solutions even when the obstacle is a nice bounded sets with finite perimeter.
On the other hand, on manifolds with cylindrical ends, the least area problem with bounded smooth obstacle does actually admits solutions, however it is impossible to select among them a bounded area minimising envelope with maximal volume (see Example~\ref{example-cylindrical}). This is not anecdotal, as we are going to prove in Theorem~\ref{main-hull} that the existence of a {bounded} strictly outward minimising hull $\Omega^*$ of a set with finite perimeter $\Omega$ is indeed equivalent to the solvability of the {\em maximal volume-least area problem} with obstacle $\Omega$. Each of these instances -- when holding for every bounded set with finite perimeter -- is in turn equivalent to the existence of an exhausting sequence of bounded strictly outward minimising sets. 
Though the latter condition can be readily checked on a given explicit metric, it is not {\em a priori} clear if it yields an effective criterion to decide whether a relevant class of manifolds (e.g., manifolds subject to some natural geometric requirement, such as curvature bounds, controlled volume growth, etc.) admits a well posed notion of {\em strictly outward minimising hull}.
Our first main result answer this question in the affirmative, identifying a couple of simple geometric conditions under which the set $\Om^*$ is always well defined (up to a zero-measure modification) and can be characterised either as the least volume strictly outward minimising set enveloping $\Om$, or as the maximal volume solution to the least area problem with obstacle $\Om$.

\begin{theorem}
\label{main-limit-1}
Let $(M, g)$ be a complete noncompact Riemannian manifold satisfying at least one of the following two conditions:
\begin{enumerate}
\item[$(i)$] \emph{(Euclidean-like Isoperimetric Inequality).} There exists a positive constant $C_{\rm iso} > 0$ such that 
\begin{equation}
\label{iso-cond-intro}
\frac{\abs{\partial \Omega}^{n}}{\abs{\Omega}^{n-1}} \geq \mathrm{C}_{\rm iso}
\end{equation} 
for any bounded set $\Omega$ with smooth boundary.

\smallskip

\item[$(ii)$] \emph{($\rm{Ric} \geq 0$ \& Polynomial Uniform Superlinear Volume Growth).} $(M, g)$ has nonnegative Ricci curvature and there exist constants $\mathrm{C}_{\mathrm{vol}} > 0$ and $1 < b \leq n$ such that
\begin{equation}
\label{minerbe}
\mathrm{C}_{\rm vol}^{-1} \, r^{b} \leq \abs{B(O, r)} \leq \mathrm{C}_{\rm vol} \, r^b
\end{equation}
holds true for any large enough  $r$.
\end{enumerate}
Then, every  bounded open set $\Omega \subset M$ with finite perimeter admits a \emph{strictly outward minimising hull} $\Omega^*$ in the sense of Definition~\ref{smh-def}. Moreover, the set $\Omega^*$ is an open bounded maximal volume solution to the least area problem with obstacle $\Om$, according to Definition~\ref{def-obst-pb}.
\end{theorem}

Before venturing in a list of significant Riemannian manifolds satisfying assumptions $(i)$ or $(ii)$, we would like to underline that -- quite surprisingly -- a complete proof of the above result does not seem to be available in the literature, even in the case where the ambient manifold is the flat $\R^n$. In that framework, the closest study is the one performed in~\cite{bassanezi-tamanini}, where a similar result was derived for the so called \emph{minimal hull}. As already mentioned, this notion differs subtly but fundamentally from the one of strictly outward minimising hull, since, in the present terminology, the minimal hull coincides with the smallest outward minimising envelope of $\Omega$, whereas we are mainly interested in the smallest \emph{strictly} outward minimising envelope of $\Omega$. Despite
the perimeter of the resulting sets will always coincide, their volumes will in general be different. 
However, we acknowledge that some of the techniques presented~\cite{bassanezi-tamanini} will be very useful in our arguments.


Going beyond the easy but fundamental example of the flat $\R^n$, Theorem~\ref{main-limit-1} encompasses a vast diversity of manifolds. First of all, let us remark that complete noncompact manifolds with nonnegative Ricci curvature and Euclidean volume growth satisfy both conditions~\eqref{iso-cond-intro} and~\eqref{minerbe}, as it was likely realised for the first time in~\cite{varopoulos-semigroups} (see also~\cite[Theorem 8.4]{hebey} for an explicit derivation through the techniques of~\cite{carron-thesis}). As this class of manifolds is particularly adapted to our methods, we will be able to push the analysis a little further, obtaining an explicit characterisation of the optimal isoperimetric constant in Theorem~\ref{isoperimetric-theorem} below. It is worth clarifying at this point that, as far as one is interested in applying Theorem~\ref{main-limit-1}, the explicit knowledge of the sharp constant in these Isoperimetric Inequalities is not required. However, it would be interesting to investigate if one could reply the above mentioned boot-strap scheme for the precise description of the isoperimetric constant also in other frameworks. For example, we recall that a positive isoperimetric constant is known to exist in Cartan-Hadamard manifolds, that are simply connected Riemannian manifolds with nonpositive sectional curvature. This is substantially a consequence of \cite{hoffman-spruck} (see  also~\cite[Theorem 8.3]{hebey}). Thus, by Theorem \ref{main-limit-1}, $\Omega^*$ is characterised also in this case as an open bounded maximal volume solution to the least area problem with obstacle  $\Omega$. This fact has been applied, for example, in the proof of~\cite[Theorem 1.3]{schulze1}, where the sharp constant in the Isoperimetric Inequality is explicitely computed for $3$-dimensional Cartan-Hadamard manifolds, reproving an earlier important result by B. Kleiner~\cite{kleiner-isoperimetric}. 

Condition \eqref{iso-cond-intro} is also readily checked on Riemannian manifolds with an explicit $\mathscr{C}^0$-description at infinity. Indeed, if $(M ,g)$ satisfies an Isoperimetric Inequality outside some compact set $K \subset M$, then the whole Riemannian manifold $(M, g)$ admits a positive isoperimetric constant. This is proved in details in~\cite[Theorem 3.2]{pigola-setti-troyanov}. 
As consequence, our Theorem applies to the important classes of Asymptotically (Locally) Euclidean manifolds and Asymptotically (Locally) Hyperbolic manifolds. 

We pass now to examine examples of manifolds obeying the condition $(ii)$ in Theorem~\ref{main-limit-1}. Observe first that this set of assumptions is naturally satisfied by complete noncompact Riemannian manifolds with nonnegative Ricci curvature that are asymptotic to warped products of the form $\dd \rho \otimes \! \dd\rho + \rho^{2\alpha} g_{N}$ with nonnegative Ricci curvature. These warped product metrics arise as models in the celebrated \cite{Che-Cold}.
Moreover, $(ii)$ encompasses the case of the so called ALF and ALG manifolds, arising as models of gravitational instantons. These manifolds, originally introduced by Hawking in \cite{hawking_instantons}, are widely studied for reasons lying at the intersection of Riemannian Geometry, Analysis and Mathematical Physics. Dropping any attempt to be complete, we address the interested reader to~\cite{minerbe2} for a nice mathematical presentation of this subject, as well as to the classification results contained in~\cite{chen-instantons1, chen-instantons2, chen-instantons3} and in the PhD thesis~\cite{chen-instantons}.

Concerning the relation of the assumption $(i)$ with $(ii)$, we observe that the existence of a global Isoperimetric Inequality is excluded if the assumption $(ii)$ is in force with $b < n$. Indeed, if \eqref{iso-cond-intro} holds, then by the arguments in~\cite{carron-iso-eucl}
one can conclude that the volume of geodesic balls $\abs{B(x, r)}$ grows at least as $r^n$ (see also \cite[Proposition 3.1]{pigola-setti-troyanov} for a self contained proof).

\subsection{$p$-capacities of the obstacles and perimeter of the hulls.} After having established  Theorem~\ref{main-limit-1}, we pass to examine the relationship between the perimeter of $\Omega^*$ and the  $p$-capacity of $\Omega$, for $p$ close to $1$. The link between these two notions emerged explicitly and naturally 
in~\cite[Theorem 5.6]{Ago_Fog_Maz_2}, as a fundamental step in the proof of the Extended Minkowski Inequality. There it was shown that in $\R^n$ the $p$-capacity of a bounded set with smooth boundary $\Omega$ approximates the perimeter of $\Omega^*$ as $p \to 1^+$. Understanding to what extent this phenomenon remains true on more general Riemannian manifolds would pave the way to  natural extensions of fundamental geometric inequalities.

On this regard it is worth noticing that without suitable assumptions on either the curvature or the volume growth  of the underlying manifold, the convergence of the $p$-capacities of $\Omega$ to the perimeter of $\Omega^*$ is not ensured, even in frameworks where the existence of a bounded strictly outward minimising hull is guaranteed. In Example~\ref{cigar} it is observed that any bounded subset $\Omega$ with finite perimeter of the higher dimensional version of the Hamilton's cigar~\cite{hamilton-cigar} admits a bounded strictly outward minimising hull. On the other hand the $p$-capacity of $\Omega$ is easily seen to vanish for every $p > 1$. The manifold in this example has nonnegative Ricci curvature and \emph{linear} volume growth. In particular, it does not satisfy neither $(i)$ nor $(ii)$ in a sharp way, since the exponent $b$ in~\eqref{minerbe} is not allowed to be $1$. Actually, we prove that if either $(i)$ or $(ii)$ is satisfied then the claimed convergence takes place for any bounded set with $\mathscr{C}^{1, \alpha}$-boundary. This is the content of our second main result. 
 
\begin{theorem}
\label{main-limit-2}
Let $(M, g)$ be a complete noncompact Riemannian manifold satisfying either assumption $(i)$ or $(ii)$ in Theorem \ref{main-limit-1} and let $\Omega \subset M$ be an open bounded subset with  $\mathscr{C}^{1, \alpha}$-boundary. Then
\begin{equation}
\label{limit-th}
\lim_{p \to 1^+} \capa_p(\Omega) = \abs{\partial \Omega^*}.
\end{equation}  
\end{theorem}
 
The above result is a fundamental tool in the forthcoming~\cite{benatti-fog-maz}, where it is applied in order to deduce the following Minkowski inequality on asymptotically conical manifold with nonnegative Ricci curvature
\begin{equation}\label{minkowski}
\left(\frac{\abs{\partial \Omega^*}}{\abs{\Sf^{n-1}}}\right)^{\frac{n-2}{n-1}}\AVR^{\frac{1}{n-1}}\leq \frac{1}{\abs{\Sf^{n-1}}}\int\limits_{\partial \Omega} \left|{ \frac{\HH}{n-1}}\right| \dd \sigma,
\end{equation}
where with $\AVR$ we denote the Asymptotic Volume Ratio of $(M,g)$ defined as
\begin{equation}
\label{avr}
\AVR \, = \, \lim_{r \to + \infty} \frac{\abs{B(x, r)}}{\abs{\mathbb{B}^n} r^n}.
\end{equation}

\subsection{Sharp Isoperimetric Inequality on manifolds with nonnegative Ricci curvature.} The last result we present consists in an application of the existence theory of bounded strictly outward minimising hulls to isoperimetry. More precisely, we show that on a manifold with nonnegative Ricci curvature and Euclidean volume growth an exhaustion of outward minimising sets allows to combine the Willmore-type inequality discovered in \cite{Ago_Fog_Maz_1} with arguments pioneered in \cite{kleiner-isoperimetric} in order to achieve the \emph{sharp isoperimetric constant} on these manifolds, up to dimension $7$. This generalises to higher dimensions the three-dimensional situation foreseen and sketched by Huisken \cite{Hui_video} and fully proven in \cite{Ago_Fog_Maz_1}.  
 
 \begin{theorem}[Isoperimetric inequality on manifolds with nonnegative Ricci curvature]
\label{isoperimetric-theorem}
Let $(M, g)$ be a complete noncompact Riemannian manifold with nonnegative Ricci curvature and Euclidean volume growth, of dimension $3 \leq n \leq 7$.
Then, 
\begin{equation}
\label{isoperimetricf}
\inf \frac{\abs{\partial \Omega}^n}{\abs{\Omega}^{n-1}} \, = \left(\frac{\abs{\Sf^{n-1}}}{\abs{\mathbb{B}^n}}\right)^{\!n-1} \!\!\!\inf \int\limits_{\partial \Omega} \left\vert \frac{\HH}{n-1} \right\vert^{n-1} \!\!\dd\sigma  \, = \, \AVR \, \frac{\abs{\Sf^{n-1}}^{n}}{\abs{\mathbb{B}^n}^{n-1}} \, ,
\end{equation}
where the infima are taken over bounded open subsets $\Omega \subset M$ with smooth boundary.
In particular, the sharp isoperimetric inequality
\begin{equation}
\label{isoperimetric-inequality-intro}
\frac{\abs{\partial \Omega}^{n}}{\abs{\Omega}^{n-1}} \geq \AVR \frac{\abs{\Sf^{n-1}}^n}{\abs{\mathbb{B}^n}^{n-1}}
\end{equation}
holds for any bounded $\Omega \subset M$ with smooth boundary. Moreover, equality is achieved in \eqref{isoperimetric-inequality-intro} for some bounded $\Omega \subset M$ with smooth boundary if and only if $(M, g)$ is isometric to flat $\R^n$ and $\Omega$ is a metric ball.
\end{theorem} 

The proof of the above result heavily relies on the existence of an exhaustion of $M$ by bounded subsets with mean-convex boundaries. In analogy with what has been observed for strictly outward minimising exhaustions, it is a difficult task to decide a priori whether a manifold admits or not such a structure.  

This is how Theorem \ref{main-limit-1} will enter the argument, since it will suffice to take the (strictly) outward minimising hull of any smooth enough exhausting sequence of bounded subsets to produce the desired mean-convex exhaustion. Indeed, these hulls are immediately seen to be mean-convex by means of a standard variational argument.
Observe that \eqref{isoperimetricf} actually yields more than just providing an isoperimetric inequality that is sharp on \emph{any} of the manifolds satisfying the assumptions above, in that it also shows the sharpness of the Willmore-type inequality proved in~\cite{Ago_Fog_Maz_1}. In that paper, the authors were able to prove the second equality in \eqref{isoperimetricf} by exhibiting an explicit minimising sequence under the additional assumption of \emph{quadratic curvature decay}. The above result in particular implies that such assumption can be removed at least if the dimension is smaller than eight. 
Moreover, our argument easily combines with the Bishop-Gromov's Theorem and an asymptotically optimal Isoperimetric Inequality \cite{berard-meyer} to infer the  strong rigidity statement. 

A classical consequence of the Isoperimetric Inequality is the Faber-Krahn inequality for the first Dirichlet eigenvalue of the Laplacian.
The derivation of such estimate is easily shown to hold up also  in a manifold $(M, g)$ of dimension $3 \leq n \leq 7$ with nonnegative Ricci curvature and Euclidean volume growth, where it yields
\begin{equation}
\label{faberf-intro}
\lambda_1(\Omega) \geq \AVR^{2/n} \lambda_1(\mathbb{B}_v)
\end{equation} 
for any bounded $\Omega \subset M$ with smooth boundary.
In the above formula, $\mathbb{B}_v$ denotes a metric ball in flat $\R^n$ of the same volume as $\Omega$ in $(M, g)$,
while $\lambda_1(\Omega)$ and $\lambda_1(\mathbb{B}_v)$ denote the first eigenvalue of $\Omega$ and $\mathbb{B}_v$ with respect to $g$ and the flat metric of $\R^n$ respectively. A rigidity statement characterising the equality case in \eqref{faberf-intro} will thus be derived from that of Theorem \ref{isoperimetric-theorem}.
 This will be the content of Theorem \ref{faber-th} below.

%
%
%
%

\subsection*{Summary.} This paper is structured as follows. Section \ref{sec:strictly-outward} is mostly devoted to study the relation, in a general complete noncompact Riemannian manifold, between the existence of a maximal volume solution to the least area problem with obstacle and the notion of the strictly outward minimising hull. In Section \ref{sec:families}, we prove  Theorem \ref{main-limit-1}. Here, precisely in Subsection \ref{sub-imcf}, we recall the notion of Weak Inverse Mean Curvature Flow, and detail some of the main applications in the present framework. More precisely, we will be able to get, in addition to its very definition $(A)$ and its variational characterisation $(B)$ also the third geometric characterisation $(C)$ of the strictly outward minimising hull as the interior of the $0$-sublevel set of the arrival time function for the weak IMCF starting at $\pa\Omega$ (see Proposition~\ref{imcf-hulls}). In Section \ref{sec:p-lim} we prove Theorem \ref{main-limit-2}. Finally, in Section \ref{sec:isoperimetric}, we prove Theorem \ref{isoperimetric-theorem}.

The paper ends with two Appendices. In the first one we prove an existence and uniqueness theorem for the $p$-capacitary potential of a bounded set $\Omega$ with $\mathscr{C}^{1, \alpha}$-boundary sitting inside a complete noncompact  manifold admitting a positive $p$-Green's function that vanishes at infinity, that is auxiliary in the proof of Theorem \ref{main-limit-2}.  In the second one, we provide a proof of a slightly more general version of the so called P\'{o}lya-Szeg\"o Principle, allowing us to deduce the Faber-Krahn inequality from Theorem \ref{isoperimetric-theorem}.

}

%
%

\bigskip

\begin{center}
{\emph{Added note}}
\end{center}
During the redaction of the manuscript, it was pointed out to the authors that the Isoperimetric Inequality resulting from Theorem~\ref{isoperimetric-theorem} already appeared in the very recent preprint by S. Brendle~\cite{brendle-isoperimetric}. The two proofs are however clearly different in spirit and the results as well have different features. Both theorems are  extensions to higher dimension of the $3$-dimensional case, proven in~~\cite[Theorem 1.8]{Ago_Fog_Maz_1}. The most evident advantage of Brendle's approach is that it is not subject to any dimensional restriction. On the other hand, using our methods, it is quite easy to provide the Isoperimetric Inequality with the natural rigidity statement, highlighting at the same time the relationship between the optimal isoperimetric constant and the infimum of the Willmore-type functional.

\bigskip

\textbf{Acknowledgements.} The authors are grateful to G. Antonelli and G. P. Leonardi for the very helpful discussions had 
during the preparation of the manuscript, to L. Mari for having clarified to the authors some aspects of \cite{mari-rigoli-setti}, to E. Spadaro for having brought to their attention the paper \cite{mondino-spadaro} and discussed with them an important result in \cite{focardi-spadaro}, and to D. Semola for having pointed out the preprint \cite{brendle-isoperimetric}.
They also thank A. Carlotto, A. Malchiodi and G. Huisken for their interest in this work .

The authors are members of Gruppo Nazionale per l'Analisi Matematica, 
la Probabilit\`a e le loro Applicazioni (GNAMPA), which is part of the 
Istituto Nazionale di Alta Matematica (INdAM),
and they are partially funded by the GNAMPA project 
``Aspetti geometrici in teoria del potenziale 
lineare e nonlineare''.

\section{The strictly outward minimising hull of a set with finite perimeter}
\label{sec:strictly-outward}

\subsection{Preliminaries}
\label{prel-perimeter}
In stating the following main concepts in the theory of finite perimeter sets in Riemannian manifolds, we are adapting the discussion found in \cite[Section 1]{miranda-bv} about BV-functions. We indicate also \cite[Section 1]{ambrosio-magnani}
for a more general introduction, that in particular applies to our setting.

Let $A \subseteq M$ be an open set. Then, we let $P(E, A)$ denote the De Giorgi's \emph{relative perimeter} of the measurable set $E$ in $A$. It is defined as 
\begin{equation}
\label{perimeter}
P(E, A) = \sup\left\{\int_E \dive\, T \dd\mu \quad \Big\vert \quad T \in \Gamma_c(TA) , \, \,\sup_A \, \abs{T} \leq 1\right\},
\end{equation}
where 
$T \in \Gamma_c(TA)$ if it is a section of the tangent bundle $TA$ of $A$ with compact support.
When $A = M$, we just talk about the perimeter of $E$ and we denote it by $P(E)$. 

If \eqref{perimeter} is finite for any bounded $A \subset M$ then, for any $\omega \in \Gamma_c(T^*M)$, we have that $(\D \chi_E, \omega)$ defines a Radon measure, where $\D \chi_E \in (\Gamma_c(T^*M))'$ denotes the distributional gradient of $\chi_E$, the characteristic function of $E$. Moreover, the total variation $\abs{\D \chi_E}$ of $\D \chi_E$ my be identified with the perimeter of $E$ by $\abs{\D \chi_E}(A) = P(E, A)$ for any open bounded $A \subset M$.   

The De Giorgi's \emph{reduced boundary} $\partial^* E$ can be thus defined as
\begin{equation}
\label{reduced}
\partial^* E = \left\{x \in M \quad \Big\vert \quad \lim_{r \to 0^+}\frac{\D \chi_E(B(x, r))}{\abs{\D \chi_E}(B(x, r))} \,\, \text{exists and belong to} \,\, \Sf^{n-1} \right\}.
\end{equation} 
\smallskip

We define now the \emph{measure theoretic interior} of a measurable set $E$  as the points of density one for $E$, namely
\begin{equation}
\label{interior}
\text{Int} (E) = \left\{x \in M \quad \Big\vert \quad \lim_{r \to 0^+} \frac{\abs{E \cap B(x, r)}}{\abs{B(x, r)}}=1\,\right\}.
\end{equation}
It follows from Lebesgue Differentiation Theorem that 
\begin{equation}
\label{representative}
\abs{E \Delta \, \text{Int} (E)} = 0,
\end{equation}
see \cite[Theorem 5.16]{maggi}.
Importantly, 
a set with finite perimeter $E$ satisfies,
\begin{equation}
\label{topint}
\partial \,\text{Int} (E) = \overline{\partial^* E},
\end{equation}
that is, the topological boundary of the measure theoretic interior of a set with finite perimeter coincides with the closure of its reduced boundary. We address the reader to \cite[Theorem 10]{caraballo} for a proof of this nice property.

\subsection{Basic properties of outward minimising sets}
We recall, according to \cite{Hui_Ilm}, the notion of  \emph{outward} and \emph{strictly outward minimising sets}.
\begin{definition}[Outward minimising and strictly outward minimising sets]
\label{outward-def}
Let $(M, g)$ be a complete Riemannian manifold. Let $E \subset M$ be a bounded set with finite perimeter. We say that $E$ is \emph{outward minimising} if for any $F \subset M$ with $E \subseteq F$ we have $P(E) \leq P(F)$. We say that $E$ is \emph{strictly outward minimising} if it is outward minimising and any time $P(E) = P(F)$ for some $F \subset M$ with $E \subseteq F$ we have $\abs{F \setminus E} = 0$.
\end{definition}   
\begin{remark} 
\label{zero-set}
It is easy to show that the notions of outward minimising and strictly outward minimising are stable under zero-measure modification. We are frequently making use of this fact in the rest of the paper.
\end{remark}
It is well known that locally area minimising sets satisfy upper and lower density estimates. It is then not surprising to realise that outward minimising sets, that can be thought as one sided minimisers, satisfy upper density estimates. We state this fact in the following Lemma, whose proof is a straightforward adaptation of~\cite[Theorem 16.14]{maggi}. The arguments employed are easily transported to a Riemannian ambient due to their local nature.
\begin{lemma}
\label{density-lemma}
Let $(M, g)$ be a complete Riemannian manifold, and let $E \subset M$ be an outward minimising set. Then, there exists $r_0 > 0$ such that for any $0 < r < r_0$ and $x \in \overline{\partial^* E}$ there holds
\begin{equation}
\label{upper}
\frac{\abs{E \cap B(x, r)}}{\abs{B(x, r)}} \leq \mathrm{C}
\end{equation}
for some constant $0< \mathrm{C} < 1$ independent of $x$.
\end{lemma}

In \cite[Theorem 6, (i)]{caraballo}, the author showed that if a set with finite perimeter satisfies relaxed density estimates then its measure theoretic interior and exterior are open. We report (a part of) his argument, in order to clarify that the upper density estimate implies that the measure theoretic interior is an open set.
\begin{lemma}
\label{openness-lemma}
Let $(M, g)$ be a complete Riemannian manifold. Let $E \subset \R^n$ be a set with finite perimeter, and suppose that for any $x \in \overline{\partial^* E}$
\begin{equation}
\label{condition}
\lim_{r \to 0^+}\frac{\abs{E \cap B(x, r)}}{\abs{B(x, r)}} < \delta < 1
\end{equation}
uniformly on $\overline{\partial^* E}$.
Then $\mathrm{Int} \, (E)$ is open.
\end{lemma}
\begin{proof}
Let $y \in \text{Int} (E)$. First observe that $y \notin \overline{\partial^* E}$, since otherwise \eqref{condition} would contradict the condition \eqref{interior} defining $\text{Int} (E)$. We now construct a ball centered at $y$ fully contained in $\text{Int} (E)$. Let $d =\text{dist}(y, \overline{\partial^* E})$. By the definition \eqref{interior} of $\text{Int} (E)$, 
\[
\frac{\abs{E \cap B(y, r)}}{\abs{B(y, r)}} > 0
\]
for some $r' \in (0, d)$. Then we can deduce 
\begin{equation}
\label{compl-zero}
\frac{\abs{(M \setminus E \cap B(y, r')}}{\abs{B(y, r')}} = 0,
\end{equation}
since, otherwise, the relative isoperimetric inequality \cite[Proposition 12.37]{maggi} would yield $\abs{\partial^* E \cap B(y, r')} > 0$, in turn resulting in the contradiction $ \text{dist}(y, \overline{\partial^* E}) < d $. Observe that such inequality clearly holds for small balls also in Riemannian manifolds, and it can be shown just by performing the computations in a local chart. The above argument in particular works up to eventually choose a suitable, uniform $d'$ smaller than $d$. 
By \eqref{compl-zero}, for any $y \in B(y, r')$, we have
\[
\frac{\abs{E \cap B(z, r'')}}{\abs{B(z, r'')}} = 1
\]
for any $r'' \in (0, r' - \text{dist}(z, y))$. This clearly implies that $B(y, r') \subset \text{Int} (E)$.
\end{proof}
As a direct consequence of \eqref{representative}, Lemma \ref{density-lemma} and Lemma \ref{openness-lemma}, we obtain the  remarkable property of outward minimising sets of admitting an open representative.
\begin{proposition}
\label{openness-prop}
Let $(M, g)$ be a complete Riemannian manifold. Let $E \subset M$ be outward minimising. Then ${\rm Int}(E)$ is open. 
\end{proposition}
\subsection{Maximal volume solutions to the least area problem with obstacle in Riemannian manifolds}
\label{sub-maximalvol}
Let $(M, g)$ be a complete noncompact Riemannian manifold of dimension $n \geq 2$. We are interested in the existence of a \emph{bounded} open set with finite perimeter $E$ solving the least area problem with obstacle $\Omega$, where $\Omega \subset M$ is a bounded open set with finite perimeter. 
\begin{definition}
\label{def-obst-pb}
We say that a bounded set $E$ solves the \emph{least area problem with obstacle $\Omega$} if it contains $\Om$ and satisfies 
\begin{equation}
\label{obstacle-pb}
P(E) = \inf\{P(F) \quad \vert \quad \Omega \subset F, F \, \text{bounded set with finite perimeter}\} \,. 
\end{equation} 
We say that a bounded set $\hat{E}$ is a~\emph{maximal volume solution of the least area problem with obstacle $\Omega$} if it contains $\Om$ and satisfies
\begin{equation}
\label{obstacle-pb-sup}
\abs{\hat{E}} = \sup\{\abs{E} \quad \vert \quad E \,\, \text{solves the least area problem with obstacle} \,\, \Omega\} \, .
\end{equation}
\end{definition}

Let us point out the most basic relation between strictly outward minimising sets and maximal volume solutions to the least area problem with obstacle.
\begin{remark}
\label{var-rem}
First, we observe, that bounded solutions to the least area problem with obstacle $\Omega$ are outward minimising, and maximal volume solutions are strictly outward minimising.
In particular, outward minimising sets $E$ are characterised by being solutions to the least area problem with obstacle $E$ itself. Interestingly, it is immediately seen that $E$ is strictly outward minimising  if and only if it is also a maximal volume solution to the least area problem with obstacle $E$.
\end{remark}
It is worth pointing out that a solution to \eqref{obstacle-pb} could clearly be non-unique, as there exist outward minimising sets that are not strictly outward minimising. As we are going to see, the definition of a solution $\hat E$ to problem~\eqref{obstacle-pb-sup} is in some sense crafted to overcome this issue.

\smallskip

In the Euclidean case, the existence of a maximal volume solution to the least area problem with obstacle $\Omega$ is an easy consequence of the Direct Method (compare with \cite{bassanezi-tamanini} and \cite{sternberg-williams_existence}). In a general noncompact Riemannian manifold, minimising sequences could fatally be unbounded, as the following examples show. 
\begin{example}[Cuspidal manifolds]
\label{example-cuspidal}
Let $(M, g)$ be a complete noncompact Riemannian manifold whose metric splits as $g = \dd\rho \otimes \dd\rho + e^{-2\rho}g_{\Sf^{n-1}}$ on $M \setminus B =[1, + \infty) \times \Sf^{n-1}$ for some precompact set $B \subset M$. Then, the area of $\partial\{\rho < r\}$ decays as $e^{-r}$, and in particular enveloping a bounded set $\Omega$ with $\{\rho < r\}$, for bigger and bigger $r$, we see that \begin{equation}
\inf\{P(F) \quad \vert \quad \Omega \subset F, F \, \text{bounded set with finite perimeter}\} = 0.
\end{equation} 
In particular, there cannot exist a solution to the least area problem with obstacle $\Omega$. Observe also that $(M, g)$ has finite volume, and thus it does not support an Euclidean-like isoperimetric inequality by \cite[Proposition 3.1]{pigola-setti-troyanov} nor the Ricci curvature is nonnegative, and thus coherently the assumptions of Theorem \ref{main-limit-1} are not satisfied.
\end{example}
\begin{example}[Manifolds with a cylindrical end]
\label{example-cylindrical}
Let $(M, g)$ be a complete noncompact Riemannian manifold whose metric splits as $g = d\rho \otimes d\rho + g_{\Sf^{n-1}}$ on $M \setminus B = [1, + \infty) \times \Sf^{n-1}$ for  some precompact set $B \subset M$. Let $\Omega$ be any bounded set with finite perimeter containing $B$. Then, any set of the form $B \cup \{\rho < r\}$ containing $\Omega$ solves the least area problem with obstacle $\Omega$, since the sets $\{\rho = r\}$ are even area minimising, but since $\abs{\{\rho = r_2\}} = \abs{\{\rho = r_1\}}$ and $\abs{B \cup \{\rho < r_2\}} = \abs{B \cup \{\rho < r_1\}} + (r_2 - r_1)$ for any $r_2 \geq r_1 \geq 1$, we see at once that
\begin{equation}
\sup\{\abs{E} \quad \vert \quad E \,\, \text{solves the least area problem with obstacle} \,\, \Omega\} = + \infty.
\end{equation}
In particular, it cannot exist a bounded~\emph{maximal volume solution to the least area problem with obstacle $\Omega$}.
Observe that $(M, g)$ has linear volume growth, and therefore by \cite[Proposition 3.1]{pigola-setti-troyanov} the $L^1$-Euclidean-like Sobolev inequality cannot hold nor \eqref{minerbe} holds.      
\end{example}
In the following basic result, we enucleate a necessary and  sufficient condition for obtaining a maximal volume solution to the least area problem with obstacle in a complete noncompact Riemannian manifold, namely the existence of an exhausting sequence of bounded strictly outward minimising sets.
Here, and in the rest of the paper, we say that a sequence of bounded open sets $\{U_j\}_{j \in \N}$ \emph{exhausts} a complete Riemannian manifold $(M, g)$ if $U_j \subset M$ for any $j\in \N$ and
\begin{equation}
\label{exhaustion-def}
M \, = \, \bigcup_{j=1}^{\infty} U_j.
\end{equation}
Observe that, in $\R^n$, an exhausting sequence of strictly outward minimising sets are given by balls. 
\begin{theorem}[Existence of maximal volume solutions to the least area problem in Riemannian manifolds]
\label{existence}
Let $(M, g)$ be a noncompact Riemannian manifold. Then, there exists an exhausting sequence of bounded strictly outward minimising sets if and only if there exists a maximal volume solution to the least area problem with obstacle $\Omega$ for any bounded $\Omega \subset M$ with finite perimeter.  
\end{theorem}
\begin{proof}
One direction is almost trivial. Indeed if for any bounded $\Omega \subset M$ with finite perimeter there exists a maximal volume solution to the least area problem with obstacle $\Omega$, then, taking a sequence of geodesic balls $B(O, R_j)$ with $R_j \to \infty$ and considering the bounded maximal volume solution to the least area problem with obstacle $B(O, R_j)$ yields the desired sequence of strictly outward minimising sets (see the discussion after Definition~\ref{def-obst-pb}). Observe that $B(O, R_j)$ can be assumed to be of finite perimeter substantially by coarea formula, and that maximal volume solutions are obviously strictly outward minimising.
Let $\Omega \subset M$ be a bounded set with finite perimeter. Let $\{F_j\}_{j \in \N}$ be a minimising sequence for the least area problem with obstacle $\Omega$.
Define
\begin{equation}
\label{m}
m = \inf\{P(F) \quad \vert \quad \Omega \subset F, F \, \text{bounded set with finite perimeter}\}.
\end{equation}
Then, for any $\epsilon > 0$, there exists $j_\epsilon \in \N$ such that
\begin{equation}
\label{min-sequence}
m \leq P(F_j) \leq m + \epsilon
\end{equation}
for any $j \geq j_\epsilon$.
Let $U$ be an outward minimising set containing $\Omega$, that exists by assumption. 
By a standard set-theoretical property of the perimeter (see e.g.  \cite[Lemma 12.22]{maggi}), we have
\[
P(F_j \cap U) \leq P(F_j) + P(U) - P(F_j \cup U).
\]
Since $U$ is outward minimising, we have $P(U) \leq P(F_j)$, and then we deduce that 
\begin{equation}
\label{fjs}
P(F_j \cap U) \leq P(F_j). 
\end{equation}
In particular, since $\Omega \subset (F_j \cap U)$ we have
\[
m  \leq P(F_j \cap U) \leq m + \epsilon 
\]
for any $j \geq j_\epsilon$, that is, the sets $F_j \cap U$ form an equibounded minimising sequence.
Then, the Compactness Theorem \cite[Theorem 12.26]{maggi} yields a bounded set of finite perimeter $F$ such that $\chi_{F_j \cap U} \to \chi_F$ in $L^1$, possibly along a subsequence.  Moreover, by the lower semicontinuity of the perimeter \cite[Proposition 12.15]{maggi}, the set $F$ is a minimiser for the minimisation problem \eqref{obstacle-pb}. We are left to show that we can modify $F$ in order to obtain a set $E$ containing $\Omega$ with $P(E) = P(F)$. Actually, it suffices to define $E = F \cup (\Omega \setminus F)$. Clearly $\Omega \subseteq E$. Moreover, since $\Omega \subset F_j \cap U$ for any $j$, and $\chi_{F_j\cap U} \to \chi_U$ in $L^1$, we have $\abs{\Omega \setminus F}=0$. Since De Giorgi's perimeter is defined up to null sets, we have $P(E) = P(F)$.
\smallskip

Let now $\{E_j\}_{j \in \N}$ be a sequence of bounded sets with finite perimeter containing $\Omega$ with
\begin{equation}
\label{maximising}
\abs{E_j} \to \sup\{\abs{E} \quad \vert \quad E \,\, \text{solves the least area problem with obstacle} \,\, \Omega\}
\end{equation}
as $j \to + \infty$. 
In particular, $P(E_j) = m$ for any $j$.  
Let $U$ be a bounded \emph{strictly} outward minimising set containing $\Omega$. Then, as in \eqref{fjs}, we have $P(E_j \cap U) \leq P(E_j)$.
We thus have
\[
m \leq P(E_j \cap U) \leq P(E_j) = m.
\]
We deduce, by standard set operations \cite[Theorem 16.3]{maggi} that
$P(U, E_j) = P(E_j, M \setminus U)$. But then
\[
P(U \cup (E_j \setminus U)) = P(U) + P(E_j, M \setminus U) - P(U, E_j) = P(U),
\]
but then, since $U$ is strictly outward minimizing $\abs{E_j\setminus U} = 0$.
This implies that the sequence of $\{E_j\}_{j \in \N}$ is essentially bounded, and 
the Compactness Theorem for sets with finite perimeter then yields a bounded set $E'$ realising the supremum in \eqref{maximising}. Up to a zero-measure modification as before, we can also suppose $\Omega \subseteq E'$.
By the lower semicontinuity of the perimeter we also have
\[
m = \liminf_{j \to \infty} P(E_j \cap U) \geq P(E')  \geq m.
\]
We showed that $P(E') = m$ and its volume realises the supremum in \eqref{maximising}, that is, $E'$ is a maximal volume solution to the least area problem with obstacle $\Omega$.
\end{proof}

\begin{remark}
A perusal of the first part of the above proof shows that, if one is interested only in a solution to the least area problem with obstacle~\eqref{obstacle-pb}, then it is sufficient to check the existence of an exhausting sequence of outward minimising sets. This is done on the cylindrical manifolds of Example~\ref{example-cylindrical}, where the cross-sections give an exhausting sequence of outward minimising sets that are not strictly outward minimising. 
\end{remark}

\subsection{The strictly outward minimising hull of a set with finite perimeter} Aim of this section is to investigate the notion of \emph{smallest} strictly outward minimising envelope of a bounded set $\Omega$. More concretely, we consider the family of Strictly Outward Minimizing Bounded Envelopes of $\Om$
\[
{\rm{SOMBE}} \, (\Omega) = \{F\subset M \, \vert \, \Omega \subseteq F \, \mathrm{and} \, F \, \mathrm{is \, \, bounded \,\, and \,\, strictly \,\,   outward \,\,  minimising}\}
\]
and we look for solutions to the minimisation problem 
\begin{equation}
\label{smh}
\inf_{F \in \mathrm{SOMBE}(\Omega)} \abs{F} \, .
\end{equation}
We observe that as soon as a strictly outward minimising envelope of $\Om$ exists, it is readily seen that a set $E$ fulfilling 
\begin{equation}
\label{construction-hull}
|E| = \inf_{F \in \mathrm{SOMBE}(\Omega)} \abs{F} \, .
\end{equation}
can be constructed by taking a minimising sequence $\{E_j\}_{j \in \N}$ and defining  $E = \cap_{j = 1}^{\infty} E_j$. It is immediately checked that $E$ realises the infimum.
Moreover, if $F$ is another bounded set with finite perimeter realising \eqref{smh}, then clearly so does $E \cap F$, and so $\abs{E} = \abs{F} = \abs{E\cap F}$. As a consequence, we have $\abs{E \Delta F} = 0$, that is $E$ and $F$ differ by a negligible set.
\begin{definition}[Strictly outward minimising hull]
\label{smh-def}
Let $(M, g)$ be a complete Riemannian manifold, and let $\Omega \subset M$ be a bounded open set with finite perimeter. The strictly outward minimising hull of $\Om$ is the measure theoretic interior $\mathrm{Int}(E)$ of any set $E$ that solves the minimisation problem~\eqref{smh}. Such a set is well defined up to zero-measure modifications and will be denoted by $\Om^*$.
\end{definition}

\begin{remark}
In case $\mathrm{SOMBE}(\Omega)$ were empty, but there existed unbounded strictly outward minimising envelopes of $\Omega$, one should still be able to work out a suitable notion of strictly outward minimising hull, with possibly infinite volume but with finite perimeter. However, in order to prove our main results there is no real need to provide such generalisation, and this why we limit ourselves to describe \emph{bounded} strictly outward minimising hulls.
\end{remark}

\begin{remark}[Relations with other notions of hulls] 
In addition to the classical convex hull and to the strictly outward minimising hull above, several other notions of hulls, that can be interpreted as the smallest set containing a given one and fulfilling a certain property, are present in literature. 
As already briefly pointed out in the Introduction, the \emph{minimal hull} of \cite{bassanezi-tamanini}  can be directly compared with $\Omega^*$, at least within the geometry of flat $\R^n$. Indeed, one realises that such hull can be obtained by taking, in place of the  of the smallest strictly outward minimising envelope of a given set of finite perimeter $\Omega$ in the sense of \eqref{construction-hull}, the smallest outward minimising set containing $\Omega$, that will be in general only contained in $\Omega^*$. This can be effectively defined by replacing $\mathrm{SOMBE}(\Omega)$ with the collection of all the outward minimising envelopes of $\Omega$. According to \cite{bassanezi-tamanini}, the minimal hull minimises,  among the minimisers of the least area problem with obstacle $\Omega$, the Hausdorff distance from $\Omega$. The main difference with the strictly outward minimising hull lies here, since $\Omega^*$ on the other hand will be seen, in the theorem below, to maximise the volume among such solutions. This  difference will be somehow reflected in our proof, where the main issue will be showing that $\Omega^*$ is strictly outward minimising. This will easily yield the maximal volume property.

Another notion of geometric hull can be found in \cite{spadaro-hull}, where the author describes the \emph{mean-convex hull} of a set $\Omega \subset \R^n$. Such notion seems to share important similarities and differences with $\Omega^*$ above, and we think that understanding their relations and their possible applications may be very interesting. We refer the reader to the references therein for other possibly related notions, appeared in recent times.
\end{remark}

We are now ready to establish a fundamental correspondence between the notion of strictly outward minimising hull and the maximal volume solutions to the least area problem with obstacle. 
As a by-product we obtain that such a solution is essentially unique, whenever its existence is guaranteed for any obstacle.

\begin{theorem}
\label{hull-th}
Let $(M, g)$ be a complete noncompact Riemannian manifold admitting, for any bounded $\Omega \subset M$ with finite perimeter, a bounded, maximal volume solution $\hat{E}$ to the least area problem with obstacle $\Omega$~\eqref{obstacle-pb-sup}. Then the strictly outward minimising hull of $\Omega$ coincides with $\hat{E}$ up to a negligible set. 
\end{theorem}
\begin{proof}
We split the proof in three steps. 
\smallskip

\emph{Step 1.} We prove that finite intersections of strictly outward minimising sets are strictly outward minimising. It clearly suffices to prove that the intersection of two strictly outward minimising sets $E_1$ and $E_2$ is outward minimising. We do it by first showing that it is outward minimising. This preliminary passage is carried out in the proof of \cite[Proposition 1.3]{bassanezi-tamanini}, but we include the argument, since it will be important also in the sequel. Let $F \subset M$ such that $E_1 \cap E_2 \subset F$, and define $L = F\setminus (E_1 \cap E_2)$. 
Then, applying the already recalled well known property holding for sets with finite perimeter
\[
P(F \cap G) + P(F \cup G) \leq P(F) + P(G)
\]
to the couples of sets $F=(E_1 \cap E_2) \cup L$ and $G = E_1$, $F=(E_1 \cap E_2) \cup (L \cap E_1)$ and $G=E_2$, give respectively
\begin{equation}
\label{tama1}
P\Big((E_1 \cap E_2) \cup (L \cap E_1)\Big) + P(E_1 \cup L) \leq P\Big((E_1 \cap E_2) \cup L\Big)  + P(E_1)
\end{equation}
and
\begin{equation}
\label{tama2}
P(E_1 \cap E_2)  + P\Big((E_1 \cap L) \cup E_2\Big) \leq P\Big((E_1 \cap E_2) \cup (L \cap E_1)\Big) + P(E_2).
\end{equation}
Combining \eqref{tama2} with \eqref{tama1} we obtain at once the following chain of inequalities
\begin{equation}
\label{chain-tama}
\begin{split}
P(E_1 \cap E_2) &\leq P\Big((E_1 \cap E_2) \cup (L \cap E_1)\Big) + P(E_2) - P\Big((E_1 \cap L) \cup E_2\Big) \\ 
&\leq P\Big((E_1 \cap E_2) \cup (L \cap E_1)\Big) \\
&\leq P\Big((E_1 \cap E_2) \cup L\Big) + P(E_1) - P(E_1 \cup L) \\
&\leq  P\Big((E_1 \cap E_2) \cup L\Big) =P(F)
\end{split}
\end{equation}
where the second and fourth inequality are due to the property of $E_1$ and $E_2$ to be outward minimising.
Assume now that that $P(E_1 \cap E_2) = P(F)$ for some $F \supseteq E_1 \cap E_2$.  Then, the second and the last inequality in \eqref{chain-tama} become equalities, yielding respectively
\begin{align}
P(E_2) &= P\Big((E_1 \cap L) \cup E_2\Big) \label{eq-tama1} \\
P(E_1) &= P(E_1 \cup L) \label{eq-tama2}.
\end{align}
Since $E_1$ is strictly outward minimising, we can deduce from \eqref{eq-tama2} that 
\begin{equation}
\label{cons2}
\abs{(E_1 \cup L) \setminus E_1} = \abs{L\setminus E_1} = 0.
\end{equation}
On the other hand, by \eqref{eq-tama1} and by $E_2$ being strictly outward minimising we obtain
\begin{equation}
\label{boh}
\left\vert{\Big(E_1 \cap L) \cup E_2\Big) \setminus E_2}\right\vert=0.
\end{equation}
By the definition of $L$ it is easily seen  that the sets $(E_1 \cap L) \cup E_2$ and $E_2$ are disjoint, and then \eqref{boh} actually reads
\begin{equation}
\label{cons1}
\abs{E_1 \cap L} = 0.
\end{equation}
Combining \eqref{cons1} with \eqref{cons2} we obtain
\[
\abs{F \setminus (E_1 \cap E_2)} = \abs{L} = \abs{L \setminus E_1} + \abs{L \cap E_1} = 0,
\]
that is, $E_1 \cap E_2$ is strictly outward minimising, as claimed.
\smallskip

\emph{Step 2.} 
We now prove that countable intersections of strictly outward minimising sets are strictly outward minimising. By the construction of $\Omega^*$ discussed above, that is viable $\hat{E}$ actually provides a strictly outward minimising set containing $\Omega$, so that $\mathrm{SOME}(\Omega)$ is not empty, this will imply that $\Omega^*$ is strictly outward minimising too.
Let $\{E_i\}_{i \in \N}$ be a sequence of strictly outward minimising sets.
We want to show that $E = \bigcap_{i=1}^{\infty}  E_i$ is strictly outward minimising.
First, \cite[Proposition 1.3]{bassanezi-tamanini} ensures that $E$ is outward minimising.
Assume then that $P({E}) = P(F)$ for some $F \subset M$ containing $E$. Clearly, $F$ is  outward minimising. 
Let $F'$ be a maximal volume solution to the least area problem with obstacle $F$.
Then $F'$ is strictly outward minimising (compare with Remark \ref{var-rem}) and $P(F') = P(F)$, being $F$ outward minimising. 
Let $A_j = \bigcap_{k=1}^{j} E_k$. Since $F'$ is strictly outward minimising, by Step 1 $A_j \cap F'$ is strictly outward minimising.   In particular, since  $A_{j + 1}\cap F' \subset  A_j\cap F' \subset F'$, we have $P(A_{j+1}\cap F') \leq P(A_j \cap F') \leq P(F')$.  Then, by the lower semicontinuity of the perimeter, we have
\begin{equation}
\label{chain-strict}
P(F')=P(E \cap F') \leq \liminf_{j \to \infty} P(A_j\cap F') \leq P(A_k \cap F') \leq P(F') 
\end{equation}
for any $k \in \N$, where the last inequality is again due to the fact that $A_k$ is outward minimising. In particular $P(A_k\cap F') = P(F')$ for any $k$. Since $A_k\cap F'$ is in fact strictly outward minimising, we deduce that $\abs{F' \setminus (A_k\cap F')} = \abs{F' \setminus A_k}  = 0$ for any $k$. Since we have $\abs{{A_k} \setminus E} \to 0$ as $k \to \infty$, we conclude that $\abs{F' \setminus E} = \lim_{k \to \infty} \abs{F' \setminus A_k} = 0$. Since $F \subseteq F'$ this trivially implies $\abs{F \setminus E} = 0$.
We have shown that $E$, and thus $\Omega^*$, is strictly outward minimising, as claimed. 
\smallskip

\emph{Step 3.} We show that $\Omega^*$  differs from $\hat{E}$ by null sets, in particular revealing itself as a maximal volume solution to the least area problem with obstacle $\Omega$. thus concluding the proof.
On the one hand, observe that trivially $\abs{\Omega^* \cap \hat{E}} \leq \abs{\Omega^*}$. However, if the strict inequality sign were in force, then this would contradict \eqref{smh}. We deduce then that $\abs{\Omega^* \setminus \hat{E}} = 0$.
On the other hand, since $\Omega^*$ is strictly outward minimising by Step 2, and $\hat{E}$ is strictly outward minimising too, we have by Step 1 $P(\Omega^* \cap \hat{E}) \leq P(\hat{E})$. However, since $\hat{E}$ solves the least area problem with obstacle $\Omega$, we actually have $P(\Omega^* \cap \hat{E}) = P(\hat{E})$. But then, since $\Omega^* \cap \hat{E}$ is strictly outward minimising, we deduce that $\abs{\hat{E} \setminus \Omega^*} = 0$, completing the proof.
\end{proof}
Let us summarise Theorem \ref{existence} and Theorem \ref{hull-th} in the following general statement. We recall that, since the above result in particular  ensures that $\Omega^*$ is strictly outward minimising, it is also open by Proposition \ref{openness-prop}.
\begin{theorem}
\label{main-hull}
Let $(M, g)$ be a complete noncompact Riemannian manifold. Then, the following statements are equivalent:
\begin{itemize}
\item[$(i)$] $(M, g)$ admits an exhausting sequence of bounded strictly outward minimising sets; 
\item[$(ii)$] For any bounded $\Omega \subset M$ with finite perimeter, there exists a bounded maximal volume solution $\Omega$ to the least area problem with obstacle $\Omega$;
\item[$(iii)$] For any bounded set with finite perimeter $\Omega \subset M$, the strictly outward minimising hull $\Omega^*$ is the unique bounded maximal volume solution to the least area problem with obstacle $\Omega$, up to null modifications. Moreover, $\Omega^*$ is an open set.
\end{itemize}
\end{theorem}

\begin{remark}
\label{test-out}
In force of the above result one gets a simple criterion for deciding whether a bounded set of finite perimeter $\Omega$ is outward minimising or not. Indeed, outward minimising sets are characterised as those satisfying
\begin{equation}
\label{test1}
P(\Omega) =  P(\Omega^*).
\end{equation}
Observe that, by Theorem \ref{main-hull}, we always have $P(\Omega^*) \leq P(\Omega)$. Checking condition \eqref{test1} then amounts to check that $P(\Omega) \leq P(\Omega^*)$.
\end{remark}

\subsection{Regularity of solutions to the least area problem with obstacle and applications}
We recall a fundamental regularity result for a  solution $E$ of the least area problem with obstacle $\Om$, under additional  assumptions on the regularity of $\partial \Omega$. This is the content of \cite[Section 3]{tamanini-regularity} and of \cite{sternberg-williams}. According to the comprehensive \cite[Theorem 1.3]{Hui_Ilm}, this regularity result holds true also in a general Riemannian setting. An inspection of the above proofs shows that they actually consider the representative $\mathrm{Int}(E)$.

\begin{theorem}[Regularity of the strictly outward minimising hull]
\label{regularity-hulls}
Let $(M, g)$ be a complete noncompact Riemannian manifold admitting an exhausting sequence of bounded strictly outward minimising sets. Let $\Omega \subset M$ be a bounded set with $\mathscr{C}^{1, \alpha}$ boundary. Then
\begin{itemize}
\item[$(i)$] $\partial \mathrm{Int}(E)$ is a $\mathscr{C}^{1, \alpha}$-hypersurface in a neighbourhood of $\partial \mathrm{Int}(E) \cap \partial \Omega$. If $\partial \Omega$ is $\mathscr{C}^2$, then $\partial \mathrm{Int}(E)$ is a $\mathscr{C}^{1, 1}$-hypersurface in a neighbourhood of $\partial \mathrm{Int}(E) \cap \partial \Omega$;
\item[$(ii)$] $\partial \mathrm{Int}(E)$ is locally area minimising around any point in  $\partial \mathrm{Int}(E) \setminus \partial \Omega$. In particular there exists a (possibly empty) singular set $\mathrm{Sing} \subset \partial \mathrm{Int}(E) \setminus \partial\Omega$ with Hausdorff dimension at most $n- 8$ such that $\partial\mathrm{Int}(E) \setminus \partial \Omega$ is a smooth hypersurface in a neighbourhood of any point in  $(\partial \mathrm{Int}(E) \setminus \partial \Omega) \setminus  \mathrm{Sing}$. If $n \leq 7$, then $\mathrm{Sing}$ is empty.
\end{itemize}
\end{theorem}
The $\mathscr{C}^{1, 1}$-regularity in a neighbourhood of the obstacle is essentially optimal, as observed in \cite{sternberg-williams} building on an example in \cite{sternberg-williams_existence}.
Theorem \ref{regularity-hulls} allows us to strengthen the uniqueness result for maximal volume solutions to the least area problem with obstacle contained in Theorem \ref{main-hull} in the case of a $\mathscr{C}^{1, \alpha}$-boundary. More precisely

\begin{proposition}[Enhanced uniqueness]
\label{uniqueness-int}
Let $(M, g)$ be a complete noncompact Riemannian manifold satisfying one of the equivalent conditions given in Theorem \ref{main-hull}. Let $\Omega \subset M$ be a bounded set with $\mathscr{C}^{1, \alpha}$-boundary. Then, if $\hat{E}$ is a maximal volume solution the least area problem with obstacle $\Omega$, we have $\Omega^* = \mathrm{Int}(\hat{E})$.
\end{proposition}
\begin{proof}
By Theorem \ref{regularity-hulls}, both $\Omega^*$ and $\mathrm{Int}(\hat{E})$ have boundaries with null volume measure. Moreover, by Proposition \ref{openness-prop}, they are open sets. Thus, since $\Omega^*$ coincides with $\hat{E}$ modulo null sets we have $\abs{\Omega^* \setminus \mathrm{Int}(\hat{E})} = \abs{\Omega^* \setminus \overline{\mathrm{Int}({\hat{E}})}} = 0$. Since $\Omega^* \setminus \overline{\mathrm{Int}({\hat{E}})}$ is an open set, this can only occur if $\Omega^* \subseteq {\hat{E}}$. Arguing the other way round we conclude.
\end{proof}

We find convenient to recall here the notion of weak mean curvature and weak mean-convexity, that, due to the above regularity theorem, becomes very natural for least area sets in presence of a smooth obstacle. 
\begin{remark}[Weak mean curvature and weak mean-convexity]
\label{weak-mean}
We point out that, by a very standard variational argument, the weak mean curvature  of $\partial \Omega^*$ is nonnegative. A notion of variational weak mean curvature is adopted in \cite[Section 1]{Hui_Ilm}, that in particular can be used in our context. However, since in this paper (precisely in Section \ref{sec:isoperimetric}) we are going to deal with this notion only in reference to a $\mathscr{C}^{1, 1}$ hypersurface, we just call weak mean curvature the one naturally defined almost everywhere in this situation, and that belongs to $L^{\infty}$. Similarly, we say that a $\mathscr{C}^{1, 1}$-hypersurface is weakly mean-convex if its weak mean curvature is nonnegative almost everywhere. With this terminology solutions to the least area problem in presence of a smooth obstacle are weakly mean-convex.
\end{remark}

We finally state the following nice \emph{one sided approximation} result due to Schmidt \cite{schmidt}, that we are going to use in the $p$-capacitary approximation of $\abs{\partial \Omega^*}$. It asserts that a bounded set with finite perimeter admits a one sided approximation in perimeter by bounded sets with smooth boundary if $P(\Omega) = \abs{\partial \Omega}$. This is clearly the case for $\Omega^*$,  if $\partial \Omega$ is assumed to be $\mathscr{C}^{1, \alpha}$, due to Theorem \ref{regularity-hulls}. The arguments being purely local, Schmidt's result applies straightforwardly in Riemannian manifolds. Observe also that although in \cite{schmidt} just interior approximation is worked out, an analogous exterior approximation obviously follows as well.
\begin{theorem}[Exterior approximation of $\Omega^*$]
\label{schmidt}
Let $(M, g)$ be a complete noncompact Riemannian manifold admitting an exhausting sequence of bounded strictly outward minimising sets. Let $\Omega\subset M$ be an open bounded set with $\mathscr{C}^{1, \alpha}$-boundary, and $\Omega^*$ be its strictly outward minimising hull. Then, there exists a 
sequence of bounded sets $\{\Omega_k\}_{k \in \N}$ with smooth boundary such that 
\begin{equation}
\label{apprf}
 \Omega^* \subset \Omega_k, \quad \abs{\partial\Omega_k} \to \abs{\partial\Omega^*}.
 \end{equation}
\end{theorem} 

\section{Families of manifolds admitting a bounded strictly outward minimising hull} 
\label{sec:families}
Here, we show that under the assumptions of Theorem \ref{main-limit-1} the notion of strictly outward minimising hull recalled in Definition \ref{smh-def} gives an open bounded maximal volume solution to the least area problem with obstacle. 
\subsection{Euclidean-like isoperimetric inequality and strictly outward minimising hull}
The following result affirms that a maximal volume solution to the least area problem with obstacle exists on any Riemannian manifold where an Euclidean-like Isoperimetric Inequality is available, that is, satisfying \eqref{iso-cond-intro} for any bounded $\Omega \subset M$ with smooth boundary. In particular, by Theorem \ref{existence} any bounded $\Omega \subset M$ with finite perimeter admits a strictly outward minimising hull. 
\begin{proposition}
\label{existence-iso}
Let $(M, g)$ be a complete noncompact Riemannian manifold satisfying admitting a constant $\mathrm{C}_{\rm iso} > 0$ such that
\begin{equation}
\label{iso-cond-sec}
\frac{\abs{\partial \Omega}^{n/(n-1)}}{\abs{\Omega}} \geq \mathrm{C}_{\rm iso}
\end{equation} 
for any bounded set $\Omega$ with smooth boundary. Then, for any bounded $\Omega \subset M$ with finite perimeter there exists a maximal volume solution to the least area problem with obstacle $\Omega$.
\end{proposition}
\begin{proof}
The main goal is showing that the Isoperimetric Inequality \eqref{iso-cond-sec} forces a minimising sequence for the least area problem to stay uniformly inside a ball, and then to conclude by compactness and lower semicontinuity as in the proof of Theorem \ref{existence}. Let then $F_j$ be a minimising sequence for the least area problem with obstacle made of bounded sets with finite perimeter. Clearly, we can suppose $P(F_j) \leq P(\Omega)$, and thus the well known \emph{local} compactness-lower semicontinuity properties of the perimeter (namely, the combination of \cite[Corollary 12.17]{maggi} with \cite[Proposition 12.15]{maggi}) yields a set $F$ such that a subsequence (relabeled as usual) of the $F_j$'s locally converges in $L^1$ to $F$
and
\[
P(F) \leq \liminf_{j \to + \infty} P(F_j),
\]
so that $F$ realises the infimum in the least area problem with obstacle $\Omega$. 
We claim that, up to null sets, $F \subset B(O, R)$ for some $R > 0$. Assume then by contradiction $\abs{F \setminus B(O, r_j)} > 0$ for some sequence $r_j \to + \infty$ as $j \to +\infty$. In particular, this implies that $\abs{F \setminus B(O, r)} > 0 $ for any $r$ big enough.
The assumed Isoperimetric Inequality (clearly in force, by approximation, for any set with finite perimeter) 
yields 
\begin{equation}
\label{isoperimetric-argument}
\mathrm{C}_{\mathrm{iso}}^{\frac{n-1}{n}} \abs{F \setminus B(O, r)}^{\frac{n-1}{n}} \, \leq \, P(F \setminus B(O, r)) = P(F , \overline{B(O, r)}^{\,c}) + \abs{\partial B(O, r) \cap F}
\end{equation}
for any $r$ such that $\partial B(O, r)$ has a $(n-1)$-negligible singular set.  In particular, \eqref{isoperimetric-argument} holds for almost any $r > 0$. In the above inequality, as in the following arguments, we are possibly considering a representative for $F$.
Observe that by coarea formula the function $m(r) = \abs{F \setminus B(O, r)}$ is absolutely continuous, with derivative $m'(r) = - \abs{\partial B(O, r) \cap F}$ for almost every $r > 0$. We claim that  
\begin{equation}
\label{claim-m'}
P(F , \overline{B(O, r)}^c) \leq \abs{\partial B(O, r) \cap F},
\end{equation} 
in order to deduce from \eqref{isoperimetric-argument} a differential inequality leading in turn to a contradiction. This is accomplished through a quite classical argument, relying on the minimising property of $F$. Let $L = F \setminus \overline{B(O, r)}^{\,c}$, and observe that, for $r$ large enough, we have $\Omega \subset L$. By the minimising property of $F$, we thus get, for $s < r$
\[
P(F , \overline{B(O, s)}^{\,c}) + P(F, B(O, s)) = P(F) \leq P(L) = P(L , \overline{B(O, s)}^{\,c}) + P(L, B(O, s)).
\]
Since clearly $P(L, B(O, s)) = P(F, B(O, s))$, we deduce
\begin{equation}
\label{ls}
P(F , \overline{B(O, s)}^{\,c}) \leq  P(L , \overline{B(O, s)}^{\,c}).
\end{equation}
On the other hand, observe that, as usual up to representatives
\[
 P(L , \overline{B(O, s)}^{\,c}) = P(F, B(O, r) \setminus \overline{B(O, s)}) + \abs{\partial B(O, r) \cap F}.  
\]
Letting $s \to r^-$,  the first term in the right hand side above vanishes and thus plugging this information into \eqref{ls} leaves us with the claimed \eqref{claim-m'}. Inserting \eqref{claim-m'} into \eqref{isoperimetric-argument} yields, as explained above, the differential inequality
\begin{equation}
\label{differential-isoperimetric}
\mathrm{C}_{\mathrm{iso}}^{\frac{n-1}{n}} m(r)^{\frac{n-1}{n}} \leq - 2 m'(r),
\end{equation}
holding true for almost any $r$ big enough, with $m(r) = \abs{F \setminus B(O, r)}$. Integrating it,
we get
\[
\frac{1}{2}{\mathrm{C}_{\mathrm{iso}}^{\,{(n-1)}/{n}}} (r_2 - r_1) \leq n \left[m (r_1)^{\frac{1}{n}} - m(r_2)^{\frac{1}{n}}\right]
\]
for any $r_2 > r_1$ big enough. Letting $r_2 \to + \infty$ the right hand side converges to $n m(r_1)^{1/n}$, since the other summand vanishes by the property of $F$ to have finite volume, while the left hand side cleary diverges. This contradiction arose from $\abs{F \setminus B(O, r_j)} > 0$ for some diverging sequence of $r_j's$, and thus we proved that there exists a ball $B(O, R)$ containing, up to a null set, the solution $F$ to the least area problem with obstacle $\Omega$. 
\smallskip

It remains to show that that we can find a bounded maximal volume solution to the least area problem. To see this, let $E_j$ be a maximising (for the volume) sequence of bounded solutions to the least area problem with obstacle $\Omega$. Then, by the compactness property used in the first part of this proof, such a sequence converges locally in $L^1$ to a set with finite perimeter $E$, that by lower semicontinuity is still a possibly unbounded solution to the least area problem with obstacle $\Omega$. The same argument as above involving the Isoperimetric Inequality then shows that, up a to null sets, $E$ is contained in some ball, completing the proof.
\end{proof}
For future reference, let us state the following immediate corollary of Proposition \ref{existence-iso} and Theorem \ref{main-hull}.
\begin{corollary}
\label{existence-iso-cor}
Let $(M, g)$ be a complete noncompact Riemannian manifold satisfying the assumption ${(i)}$ in the statement of Theorem \ref{main-limit-1}. Then, for any bounded $\Omega \subset M$ with finite perimeter, the strictly outward minimising hull $\Omega^*$ is an open bounded maximal volume solution to the least area problem with obstacle $\Omega$.
\end{corollary}
\subsection{Weak Inverse Mean Curvature Flow and strictly outward minimising hulls}
\label{sub-imcf}
In this subsection we show that the existence of a \emph{proper} solution to the Weak Inverse Mean Curvature Flow starting at a bounded set with smooth boundary $\Omega$ naturally yields an exhausting sequence of strictly outward minimising sets, and that the strictly outward minimising hull of $\Omega$ is simply given by the interior of the  zero-sublevel set of such a solution. This is substantially the content of \cite[Lemma 1.4]{Hui_Ilm}, but here we want to emphasise and fully detail the relation between the properties of the strictly outward minimising hull and  the existence of the Weak Inverse Mean Curvature Flow.

Let us recall the notion of Inverse Mean Curvature Flow, as well as its weak formulation introduced by Huisken-Ilmanen in \cite{Hui_Ilm}. Let $(M, g)$ be a complete noncompact Riemannian manifold, and let $\Omega \subset M$ be a bounded subset with smooth boundary given by the immersion $F_0: \partial \Omega \to \R^n$. Assume in addition that the mean curvature of $\partial \Omega$ is strictly positive. Then, we say that the hypersurfaces $\{\partial \Omega_t\}_{t \in [0, T)}$, for some $T > 0$ are evolving by IMCF with initial datum $\partial \Omega$ if they are given by immersions $F(t, \cdot): \partial \Omega \to \R^n$ satisfying
\begin{equation}
\label{IMCFp}
\frac{\partial}{\partial t} F(t, x) = \frac{1}{\HH}(t, x) \nu (t, x), \qquad\qquad F(0, x) = F_0(x)
\end{equation}
where $\nu$ is the exterior unit normal to the hypersurface $\partial \Omega_t$ and $\HH$ is its related mean curvature. It is well known that the IMCF of a strictly mean-convex hypersurface enjoys existence in some time interval $[0, T)$, see e.g. the comprehensive \cite[Theorem 3.1]{Huis_Pold}. 

Looking at the evolving hypersurfaces $\partial\Omega_t$ as level sets $\{w = t\}$ of a smooth function, it is easily seen that $w$ must satisfy the equation
\begin{equation}
\label{imcf-eq}
\dive\left(\frac{\D w}{\abs{\D w}}\right) = \abs{\D w}
\end{equation}
in the region foliated by the evolving hypersurfaces. 
In particular, if there exists a smooth solution to \eqref{IMCFp}   made of embedded closed hypersurfaces, then it is well defined the smooth function $w$ with nonvanishing gradient solving \eqref{imcf-eq}. For ${x} \in M \setminus \overline{\Omega}$, $w({x})$ in fact returns the time the Inverse Mean Curvature Flow hits the level set $\{w = w(x)\}$.
\smallskip

The weak formulation of the IMCF actually constitutes a weak formulation of \eqref{imcf-eq}. 
A weak solution to \eqref{imcf-eq} starting from a bounded set with smooth boundary $\Omega$, that we will call with abuse of terminology (since as explained above it is actually the arrival time function) Weak Inverse Mean Curvature Flow starting at $\Omega$, will be a function $w \in \mathrm{Lip}_{\mathrm{loc}}(M)$ satisfying the following conditions.
\begin{enumerate}
\item[$(i)$] For every $v \in \mathrm{Lip}_{\mathrm{loc}}(M)$ with $\{w \neq v\} \Subset M \setminus \overline{\Omega}$ and any compact set $K\subset M \setminus \Omega$ containing $\{w \neq v\}$,
\begin{equation}
\label{wimcf1}
J_w^K(w) \leq J_w^K(v)
\end{equation}
 where 
\begin{equation}
\label{J}
J_w^K(v) = \int_K \abs{\D v} + v \abs{\D w} \dd\mu. 
\end{equation}
\item[$(ii)$] The set $\Omega$ is the $0$-sublevel set of $w$, that is
\begin{equation}
\label{0-sub}
\Omega = \{w < 0\}.
\end{equation}
\end{enumerate}

\begin{remark}[Properness of the Weak IMCF]
\label{properness}
We say that the Weak Inverse Mean Curvature Flow is \emph{proper} if $w$ is a proper function. In the rest of this note, we will always consider Weak Inverse Mean Curvature Flows that are proper.
Observe that if $w(x) \to + \infty$ as $d(O, x) \to + \infty$, then $w$ is proper. The validity of this condition is assumed as definition of properness in \cite{kotschwar}. However, there may exist a proper  Weak IMCF $w$ such that $w \not\rightarrow +\infty$ at infinity (see Example \ref{cigar} below).
\end{remark}

It is definitively convenient to rephrase $(i)$ in the definition of Weak Inverse Mean Curvature Flow in terms of the level sets of the solution.
\begin{enumerate}
\item[$(i-bis)$]
For any $t \geq 0$, the sets $\{w \leq t\}$ satisfies
\begin{equation}
\label{wimcf2}
J^K_w(\{w \leq t\}) \leq J^K_w (F)
\end{equation}
for any $F\subset M$ with locally finite perimeter satisfying $\{w \leq t\} \Delta F \Subset M \setminus \overline{\Omega}$  and any compact $K \subset M\setminus\overline{\Omega}$ containing $\{w \leq t\} \Delta F$, where
\begin{equation}
\label{J-set}
J_w^K (F) = \abs{\partial^* F \cap K} - \int_{F \cap K} \abs{\D w} \dd\mu.
\end{equation}
\end{enumerate}
We refer the reader to \cite[Lemma 1.1 and Lemma 2.2]{Hui_Ilm} and the discussion thereafter for the proof of the equivalence between conditions $(i)$ and $(i-bis)$.
One can then deduce from the fundamental Minimizing Hull Property Lemma 1.4 in \cite{Hui_Ilm} that the sets ${\rm Int}\{w \leq t\}$ are strictly outward minimising. In particular, the condition of having an exhausting sequence of strictly outward minimising sets is fulfilled every time there exists the Weak IMCF, and the analysis of the preceding section allows then to define the strictly outward minimising hull of $\Omega$. It finally turns out that $\Omega^* = {\rm Int}\{w \leq 0\}$. Let us carefully prove this fact.
\begin{proposition}
\label{imcf-hulls}
Let $(M, g)$ be a complete noncompact Riemannian manifold. Assume that for any bounded $\Omega \subset M$ with smooth boundary there exists a proper weak IMCF $w$ starting from $\Omega$, then the strictly outward minimising hull $\Omega^*$ is an open bounded maximal volume solution to the least area problem with obstacle $\Omega$. Moreover, ${{\rm Int}\{w \leq 0\} = \Omega^*}$. 
\end{proposition}
\begin{proof}
Let $w$ be the Weak IMCF emanating from the bounded open set with smooth boundary $\Omega$. By \cite[Lemma 1,4, (ii)]{Hui_Ilm}, the set $\mathrm{Int}\{w \leq 0\}$ is bounded and strictly outward minimising. In particular, if $w$ exists for any such $\Omega$, taking an exhausting sequence of bounded open sets with smooth boundary yields an exhausting sequence of bounded strictly outward minimising sets, that allows us to use Theorem \ref{main-hull} to show that $\Omega^*$ defined in Definition \ref{smh-def} is an open bounded maximal volume solution to the least area problem with obstacle $\Omega$. Consequently $\Omega^* \cap \mathrm{Int}\{w \leq 0\}$ is the intersection of two strictly outward minimising sets, and then, as proved in \emph{Step 1} in the proof of Theorem \ref{hull-th}, it is  strictly outward minimising itself. Thus, since $\Omega^*$ minimises the volume among all the strictly outward minimising envelopes of $\Omega$, we conclude that $\abs{\Omega^* \setminus (\mathrm{Int}\{w \leq 0\} \cap \Omega^*)} = 0$. Thus, up to a null modification, we have $\Omega^* \subseteq \mathrm{Int}\{w \leq 0\}$.
By properties $(i-bis)$ and $(ii)$ in the definition of the Weak IMCF recalled above, applied with $F = \Omega^*$ we get, up to a suitable choice of $K$,
\[
P(\{w \leq 0\}) \,\,\,\, - \!\!\int\limits_{{\rm Int}\{w \leq 0\} \setminus {\overline{\Omega^*}}} \abs{\D w} \dd\mu \, \leq \, P(\Omega^*),
\]  
since the boundary of $\Omega^*$ is negligible in light of the regularity statement Theorem \ref{regularity-hulls}. As $\abs{\D w}$ vanishes in the open region ${\rm Int}\{w \leq 0\} \setminus {\overline{\Omega^*}} \subset \{w = t\}$, we deduce that $P(\{w \leq 0\}) \leq P(\Omega^*)$. 
Thus $\mathrm{Int}\{w \leq 0\}$ is a solution to the least area problem with obstacle $\Omega$ and, since $\Omega^*$ is of maximal volume among solutions to the obstacle problem, we necessarily have $\abs{\{w \leq 0 \}} = \abs{\Omega^*}$.
We deduce that $\mathrm{Int}\{w \leq 0\}$ is another maximal volume solution to the least area problem with obstacle $\Omega$, and we conclude by Proposition \ref{uniqueness-int}.
\end{proof}
In \cite{mari-rigoli-setti}, the authors showed that under the assumptions of $(ii)$ in Theorem \ref{main-limit-1}, any bounded set with smooth boundary can be evolved through a proper weak solution to the Inverse Mean Curvature Flow. In particular from this fact and Proposition \ref{imcf-hulls} we get that under these assumptions $\Omega^*$ satisfies the desired properties.
\begin{corollary}
\label{imcf-hulls-cor}
Let $(M, g)$ be a complete noncompact Riemannian manifold satisfying the assumptions of ${(ii)}$ in Theorem \ref{main-limit-1}. Then, for any open bounded $\Omega$ with smooth boundary, the set $\Omega^*$ is an open bounded maximal volume solution to  the least area problem with obstacle $\Omega$.
\end{corollary}

\subsection{Other aspects of strictly outward minimising sets and IMCF}
Here, we observe that a \emph{smooth} solution to the Inverse Mean Curvature Flow of a bounded open set $\Omega$ contained in a complete Riemannian manifolds never completely leaves $\Omega^*$, if the latter is an open bounded maximal volume solution to the least area problem with obstacle $\Omega$. In particular, this holds under the assumptions of Theorem \ref{main-limit-1}. This phenomenon is substantially a direct corollary of the outward minimising property of the hypersurfaces evolving by IMCF, that was observed in \cite[Smooth Flow Lemma 2.3]{Hui_Ilm}. We add a proof of this fact, because it implies, together with the long time existence theory for strictly starshaped sets with smooth and strictly mean-convex boundary, that the latters are strictly outward minimising. A simple approximation argument involving the Mean Curvature Flow will also yield that any mean-convex strictly starshaped set in $\R^n$ is outward minimising. 

\begin{proposition}[No Escape from $\Omega^*$]
\label{noescape}
Let $(M, g)$ be a complete noncompact Riemannian manifold satisfying one of the equivalent conditions in Theorem \ref{main-hull}. Then, for $\Omega \subset M$ a bounded open set with smooth boundary, let $\{\partial \Omega_t\}_{t \in [0, T]}$ be evolving by IMCF with initial datum $\partial \Omega$. Then, if $\Omega^* \Subset \Omega_t$ for some $t \in (0, T]$, then $\Omega$ is strictly outward minimising and we have $\Omega = \Omega^*$.
\end{proposition} 
\begin{proof}
Let $w: \Omega_T \to \R$ such that $\{w = t\} = \partial \Omega_t$, so that $w$ classically satisfies the level set equation \eqref{imcf-eq}. Then
\begin{equation}
\label{chain-noescape}
\begin{split}
0 \leq \int\limits_{\Omega^* \setminus \overline{\Omega}} \abs{\D w} \dd\mu &= \int\limits_{\Omega^* \setminus \overline{\Omega}} \dive\left(\frac{\D w}{\abs{\D w}} \right)  \dd\mu \\ 
&=  \int\limits_{\partial^*\Omega^*} \left\langle \frac{\D  w}{\abs{\D w}} \, \bigg\vert \, \nu_{\partial^* \Omega^*} \right\rangle \dd\sigma - \int\limits_{\partial\Omega} \left\langle \frac{\D  w}{\abs{\D w}} \, \bigg\vert \, \frac{\D w}{\abs{\D w}} \right\rangle \dd\sigma  \\
&\leq \abs{\partial \Omega^*} - \abs{\partial \Omega} \leq 0,
\end{split}
\end{equation}
where in the second equality $\nu_{\partial^* \Omega^*}$ is the measure theoretic unit normal to the reduced boundary $\partial^* \Omega^*$ and in the last inequality we used the Divergence Theorem for sets of finite perimeter  coupled with $\abs{\partial \Omega^*} = \abs{\partial^*\Omega^*}$. In particular, \eqref{chain-noescape} implies that $\abs{\Omega^* \setminus \overline{\Omega}} = 0$. By openness, that follows from \eqref{openness-prop}, this implies that $\Omega = \Omega^*$ and in particular it is strictly outward minimising.
\end{proof}
\begin{remark}
The above result originated in an earlier version of \cite{Ago_Fog_Maz_2}, where it was proved through $p$-harmonic approximation. It has been recently re-proved, substantially with the arguments above in \cite[Lemma 1]{harvie-nonstarshaped}.
\end{remark}
The celebrated result of Gerhardt \cite{gerhardt} and Urbas \cite{urbas}, asserting that strictly starshaped sets of $\R^n$ with smooth and strictly mean-convex boundary admit an immortal solution to their Inverse Mean Curvature Flow immediately combines  with Proposition \ref{noescape} to show that these sets are in fact strictly outward minimising. Let us recall that in $\R^n$ a bounded set $\Omega$ with smooth boundary is strictly starshaped with respect to some point $x_0 \in \Omega$ if and only
\[
\left\langle x - x_0 \, \vert \, \nu\right\rangle > 0
\] 
for any $x \in \partial \Omega$, where $\nu$ is the outward unit normal to the boundary of $\Omega$.
\begin{corollary}
\label{starshaped-cor}
Let $\Omega \subset \R^n$ be a bounded strictly starshaped set with smooth strictly mean-convex boundary. Then, $\Omega$ is strictly outward minimising.
\end{corollary}
A simple approximation argument based on the Mean Curvature Flow yields that the strict mean-convexity can be relaxed to mean-convexity, that is, the mean curvature $\HH$ of $\partial \Omega$ is allowed to satisfy $\HH = 0$ on some point of $\partial \Omega$. However, in this case, we are able to show that $\Omega$ is just outward minimising.

\begin{proposition}[Starshaped mean-convex sets are outward minimising]
\label{th-star}
Let $\Omega \subset \R^n$ be a  strictly \emph{starshaped} bounded open set with smooth \emph{mean-convex} boundary. Then $\Omega$ is outward minimising. 
\end{proposition}
\begin{proof}
Let $F_0: \partial \Omega \to \R^n$ be the immersion of $\partial \Omega$ in $\R^n$. Evolve this hypersurface by Mean Curvature Flow defining time dependent immersions $F: [0, \delta)\times \partial \Omega \to \R$ satisfying 
\begin{equation}
\label{mcf}
\frac{\partial F}{\partial s}(s, x) = - \HH(F(s,x))\,\nu(s,x), \qquad\qquad F(0, x) = F_0(x),
\end{equation}
where $\nu$ is the exterior unit normal to the hypersurface $\partial \Omega_s$ given by the immersion $F(s, \cdot) : \partial \Omega \to \R^n$, and $\HH$ is its related mean curvature. 
The standard short-time existence theory for geometric evolution equations (see e.g. \cite[Theorem 3.1]{Huis_Pold}) ensures the existence of a $\delta > 0$ such that a solution $F$ to \eqref{mcf} is well-defined. In other words, we have defined a sequence of bounded open sets $\{\Omega_s\}_{s \in [0, \delta)}$ with smooth boundary evolving by Mean Curvature Flow \eqref{mcf}. It is well known (see e.g. \cite[Theorem 3.2]{Huis_Pold}) that the mean curvature of these boundaries evolves by
\begin{equation}
\label{evomean}
\frac{\partial}{\partial s} \HH = \Delta \HH + \HH\abs{\hh}^2,
\end{equation}
where by $\hh$ we denote the second fundamental form that, as the other quantities appearing in the equation above, is to be understood with respect to the evolving metric on $\partial \Omega_s$. In particular, the standard Maximum Principle for parabolic equations implies that the mean curvature of $\partial \Omega_s$ for $s \in (0, \delta)$ is strictly positive. Since, by the smoothness of the flow, the sets $\Omega_s$ are still strictly starshaped for small $s$, we can conclude by Corollary \ref{starshaped-cor} that these approximating sets $\Omega_s$ are strictly outward minimising. Observe now that since the mean curvature $\HH$ of the initial datum $\partial \Omega$ is nonnegative, the flow \eqref{mcf} is actually a shrinking flow, and thus $\Omega_s \subseteq \Omega \subseteq \Omega^*$. Then, by the minimising property of $\Omega_s$ we have $\abs{\partial\Omega_s} \leq \abs{\partial \Omega^*}$, that, upon letting $s\to 0^+$, implies $\abs{\partial\Omega} \leq \abs{\partial \Omega^*}$. This means, by Remark \ref{test-out}, that $\Omega$ is outward minimising.
\end{proof}
We conclude by pointing out that in literature there are many generalisations of Gerhardt-Urbas results in non-flat ambient manifolds, for a suitable notion of starshapedness. We mention in particular the work of Brendle-Hung-Wang \cite{Brendle} that covers a considerable diversity of warped product ambient manifolds and, outside rotationally symmetric ambients, the work of Pipoli \cite{pipoli-complex} in the Complex Hyperbolic Space. It is easy to see that, thanks to these results, the above argument shows that the theses of Corollary \ref{starshaped-cor} and Proposition \ref{th-star} hold in these ambient manifolds too. Moreover, families of domains sitting inside flat $\R^n$, not necessarily star-shaped, but still smoothly immortal, have been recently provided in \cite{harvie-nonstarshaped}. Our analysis in particular applies to these domains and show that they are strictly outward minimising.
 
\section{Convergence of $p$-capacities to $\abs{\partial \Omega^*}$}
\label{sec:p-lim}
Aim of this section is to show that on manifolds satisfying either assumptions $(i)$ or $(ii)$ in Theorem~\ref{main-limit-1} we can recover the value of $\abs{\partial \Omega^*}$ as the limit for $p \to 1^+$ of the $p$-capacity of $\partial \Omega$, according to the statement of Theorem~\ref{main-limit-2}.
Let us start recalling some notation and the basic existence result of $p$-capacitary potentials.

\subsection{$p$-nonparabolicity and $p$-capacitary potentials}
Let $(M, g)$ be a complete noncompact Riemannian manifold, and let $p \geq 1$. 
We define the $p$-capacity   of a bounded set with $\mathscr{C}^{1, \alpha}$-boundary $\Omega \subset M$ as
\begin{equation}
\label{pcap-fin}
\capa_p(\Omega) \,  = \, \inf\left\{ \int_{\R^n} \abs{ \D f}^p \dd\mu \,\, \Big| \ f \geq \chi_\Omega, \,\, f \in \mathscr{C}^\infty_0(\R^n) \right\}.
\end{equation}
Let now $p > 1$.
Then $(M, g)$ is said to be $p$-nonparabolic if there exists a \emph{positive} $p$-Green's function $G: (M\times M) \setminus \mathrm{Diag}(M) \to \R$, that is, satisfying
\begin{equation}
\label{green-def}
\int\limits_M \Big\langle \abs{\D \, G_p{(O, \cdot)}}^{p-2}\D \,G_p{(O, \cdot)} \, \big\vert \, \D \phi\Big\rangle \dd\mu =  \phi (O)
\end{equation}
for any $\phi \in \mathscr{C}^\infty_c(M)$. The relation \eqref{green-def} is the weak formulation of the equation $-\Delta_p G(O, \cdot) = \delta_O$, where $\delta_O$ is the Dirac delta centered at $O$. Moreover, when referring to the $p$-Green's function of a $p$-nonparabolic manifold we mean the minimal one.

One can show that if $M$ is $p$-nonparabolic and $G_p \to 0$ at the infinity of any end then, for any open bounded $\Omega \subset M$ with $\mathscr{C}^{1, \alpha}$-boundary there exists a unique weak  solution $u_p \in W^{1, p}(M \setminus\overline{\Omega}) $ to 
\begin{align}
\label{pbp-chap3}
\left\{\begin{array}{lll}
\Delta_p u=0 & \mbox{in} & M\setminus \overline{\Omega} \, ,\\
\quad \,u=1 & \mbox{on} & \partial\Omega \, ,\\
\, u(x)\to 0 & \mbox{as} & d(O, x)\to\infty \, ,
\end{array}\right.
\end{align}
where we can suppose $O \in \Omega$. 
Moreover, the integral of $\abs{\D u_p}^p$ on $M \setminus \overline{\Omega}$ realises the $p$-capacity. Since a complete and self-contained proof of this general result does not seem easy to find in literature, we included a proof in the Appendix, adapting the argument used for \cite{salani-brunn} together with the deep $\mathscr{C}^{1, \alpha}$-estimates holding true up to the boundary of \cite{liebermann}.

\begin{theorem}
\label{existence-potential}
Let $(M, g)$ be a complete noncompact $p$-nonparabolic Riemannian manifold. Let $\Omega \subset M$ be an open bounded set with $\mathscr{C}^{k, \alpha}$-boundary for some $k \in \N_{>0}$ and $\alpha \in (0, 1)$. Let $O \in M.$ Assume also that the $p$-Green's function $G_p$ satisfies $G_p(O, x) \to 0$ as $d(O, x) \to \infty$. Then, there exists a unique weak solution $u_p$ to \eqref{pbp-chap3} attaining the datum on $\partial \Omega$ in $\mathscr{C}^{k, \beta}$ for some $\beta \in (0, 1)$.  Moreover, it holds
\begin{equation}
\label{p-cap-u-chap3}
\int\limits_{M\setminus \overline{\Omega}} \abs{\D u_p}^p \dd\mu = \capa_p(\Omega).
\end{equation} 
\end{theorem} 

\begin{remark}
\label{ricci-existence}
It is worth pointing out that, by \cite{holopainen1}, the above general result fully describes the nonnegative Ricci curvature case in terms of growth of the volume of geodesic balls. Indeed, we know that if $(M, g)$ has nonnegative Ricci curvature and 
\begin{equation}
\label{holo-nonpar}
\int_1^{+\infty} \left(\frac{t}{\abs{B(O,  t)}}\right)^{1/(p-1)} \dd t < +\infty
\end{equation}
for any $O \in M$, then \cite[Proposition 5.10]{holopainen1} gives the $p$-nonparabolicity of $(M, g)$, and the decay estimate  for the positive $p$-Green's function $G_p$ allows to conclude that $G_p \to 0$ at infinity. This is observed in \cite[Corollary 2.6]{pucci-mari}.

On the other hand, if the integral in \eqref{holo-nonpar} diverges in a complete noncompact $(M, g)$, by \cite[Proposition 1.7]{holopainen1} $(M, g)$ is $p$-parabolic, this in particular implies that $\capa_p(\Omega) = 0$ (see e.g. \cite[(1.5)]{holopainen1}) for any bounded $\Omega \subset M$ with smooth boundary and in particular the conclusion of Theorem \ref{existence-potential} cannot hold true.

Finally we point out that if an Isoperimetric Inequality \eqref{iso-cond-intro} holds true on a complete noncompact $(M, g)$, then such manifold is $p$-nonparabolic for any $1 \, < \, p \, <\, n$ and the $p$-Green's function vanishes at infinity, as realised in \cite[Theorem 3.13]{mari-rigoli-setti}, and then Theorem \ref{existence-potential} is in force.
\end{remark}
\medskip

\subsection{Proof of Theorem~\ref{main-limit-2} and a sharp counter-example to the convergence.} We are now  in position to prove our convergence result. Namely, we are going to show that under either assumption $(i)$ or $(ii)$ in Theorem~\ref{main-limit-1} the  $p$-capacity of $\Omega$ approximates the area of $\partial \Omega^*$. 

It is interesting to observe that in \emph{Step 3} below -- which is probably the most important and original step of the whole argument -- the role played by two assumptions $(i)$ and $(ii)$ is quite different. 
In presence of an Euclidean-like Isoperimetric Inequality, an argument inspired by \cite{xu} allows to show that 
\[
\capa_1 (\Omega) \, \leq \, \mathrm{C}_{n, p} \, \capa_p (\Omega),
\]
for some constant $\mathrm{C}_{n, p}$ fulfilling $\mathrm{C}_{n, p} \to 1$ as $p\to 1^+$. 
On the other hand, if $(M, g)$ is a manifold with nonnegative Ricci curvature satisfying the polynomial superlinear uniform volume growth condition, we are still able to prove \eqref{limit-th} by exploiting a decay estimate of the $p$-Green's function of $(M, g)$ with an explicit dependence on $p$. 
This estimate originated in \cite[Proposition 5.7]{holopainen1}, where it was proved for any point on the boundary of an end, and it has  been applied in \cite{mari-rigoli-setti} together with the assumed \eqref{minerbe} to obtain an analogous inequality holding true on any point outside some compact set (see \cite[Theorem 3.8]{mari-rigoli-setti}). In this regard, we observe that being able to work out Holopainen's inequality without the restriction of lying in the boundary of an end would allow to  relax the assumption \eqref{minerbe} in \cite{mari-rigoli-setti}, and consequently in the present theory, to 
\begin{equation}
\label{superlinear-vero}
\int_1^{+\infty} \frac{t}{\abs{B(O,  t)}} \dd t < +\infty,
\end{equation}
that is, roughly speaking, a strictly superlinear volume growth assumption.

\begin{remark}
\label{minerbe-relax}
We point out that in \cite{mari-rigoli-setti} the authors
actually assumed 
\begin{equation}
\label{minerbe-vero}
\frac{\abs{B(O, t)}}{\abs{B(O, s)}} \geq \mathrm{C} \left(\frac{t}{s}\right)^b
\end{equation}
for any $t \geq s > 0$ and some constant $\mathrm{C} > 0$
 to provide the inequality described above. In turn, \eqref{minerbe-vero}, that is implied by \eqref{minerbe}, is used to invoke the technical \cite[Proposition 2.8]{minerbe-sobolev}, allowing to control the size of the bounded components of the complement of $M \setminus B(O, R)$ for big $R > 0$. In fact, the desired decay estimate on the $p$-Green's function, and in turn, the validity of Theorem \ref{main-limit-2} is ensured on any Riemannian manifold with nonnegative Ricci curvature satisfying \eqref{superlinear-vero} such that for any $R$ big enough $M \setminus B(O, R)$ does not have bounded components. 
\end{remark}
\begin{proof}[Proof of Theorem \ref{main-limit-2}]
Let $(M, g)$ be a complete noncompact Riemannian manifold satisfying $(i)$ or $(ii)$ in the statement. Then, by Corollary \ref{existence-iso-cor} and Corollary \ref{imcf-hulls-cor} respectively, for any bounded $\Omega$ with finite perimeter the strictly outward minimising hull $\Omega^*$ is an open bounded maximal volume solution to the least area problem with obstacle $\Omega$.
\smallskip

Before going on, let us recall that under the assumption of Theorem \ref{main-limit-1} those of Theorem \ref{existence-potential} are satisfied, yielding a $p$-capacitary potential $u_p$ realising the $p$-capacity of $\Omega$, as illustrated in Remark \ref{ricci-existence}. We aim at proving that
\begin{equation}
\label{chain-pcap}
\abs{\partial \Omega^*} \leq \capa_1(\Omega) \leq \liminf_{p \to 1^+} \capa_p(\Omega) \leq \limsup_{p \to 1^+} \capa_p(\Omega) \leq \abs{\partial \Omega^*}.
\end{equation}
We divide the proofs in several steps.

\medskip

\emph{Step 1.} The first and easiest inequality to show in \eqref{chain-pcap} is
\begin{equation}
\label{chain-pcap-1}
\abs{\partial \Omega^*} \leq \capa_1(\Omega).
\end{equation}
It suffices to observe that for any $f \in \mathscr{C}^\infty_c(M)$ with $f \geq \chi_\Omega$ we have, by co-area formula,
\begin{equation}
\label{1-ineq-pre}
\int\limits_{M} \abs{\D f} \dd\mu 
\,\geq\, 
\int_0^1 \abs{\{f = t\}} \dd t 
\,\geq\, 
\inf \Big\{ \abs{\partial E}\ \big|\ {\Omega} \subset E, \, \partial E \, \text{smooth}\Big\} 
\,\geq\, \abs{\partial \Omega^*}.
\end{equation} 
In particular, taking the infimum over any such $f$, we get
\eqref{chain-pcap-1}.

\medskip

\emph{Step 2.} Now, we prove that
\begin{equation}
\label{chain-pcap-2}
\limsup_{p \to 1^+} \capa_p(\Omega) \leq \abs{\partial \Omega^*}.
\end{equation}
Let $E$ be any open and bounded set in $M$ with smooth boundary such that $\Omega \subset E$. Define, for $x \in M$, the function $d_E (x) = \text{dist}(x, E)$. 
Moreover, let us introduce 
a smooth cut-off function $\chi_\epsilon$ fulfilling
\begin{equation}
\label{chi}
\begin{cases}
\,\,\chi_\epsilon(t)= 1 & \mbox{in} \,\, t < \epsilon,  
\\
\,\,-\frac{1}{\epsilon} <\dot{\chi}(t) < 0 & \mbox{in}\,\, \epsilon \leq t \leq {2}\epsilon
\\
\,\,\chi_\epsilon(t) =0 &\mbox{in} \,\, t > {2}\epsilon,
\end{cases}
\end{equation}
and let us set $\eta_\epsilon(x) = \chi_\epsilon (d_E (x))$. Choosing $\epsilon$ small enough, it is easily seen, by the regularity of $d_E$ in a neighbourhood of $E$ (first observed for $\R^n$ seemingly in \cite[Lemma 14.16]{Gil_Tru_book}, see \cite[Proposition 5.17]{mantegazza-distance} for a self contained proof in a general ambient manifold), that the function $\eta_\epsilon$ is an admissible competitor in \eqref{pcap-fin}. Then,
\[
\capa_p(\Omega) \leq \int\limits_{M} \abs{\D \eta_\epsilon}^p \dd\mu
\]
for any $p\geq 1$.
Letting $p \to 1^+$, we get
\[
\limsup_{p \to 1^+} \capa_p(\Omega) \leq \int\limits_{M} \abs{\D \eta_\epsilon} \dd\mu = \int\limits_\epsilon^{2\epsilon} \abs{\dot{\chi_\epsilon}(t)} \left\vert\{d_E = t\}\right\vert \dd t,
\]
where in the last equality we applied the coarea formula combined with the fact that $\abs{\D d_E} = 1$ in a neighbourhood of $E$.
By the Mean Value Theorem, there exist $r_\epsilon \in (\epsilon, 2\epsilon)$ such that the above right hand side satisfies
\[
\int\limits_\epsilon^{2\epsilon} \abs{\dot{\chi_\epsilon}(t)} \left\vert\{d_E = t\}\right\vert \dd t = \epsilon \abs{\dot\chi_\epsilon(r_\epsilon)}\abs{\{d_E = r_\epsilon\}} <\abs{\{d_E = r_\epsilon\}},
\]
where the last inequality is due to the second condition in \eqref{chi}. Since,
as $r_\epsilon \to 0^+$, we clearly have
\[
\abs{\{d_E = r_\epsilon\}} \to \abs{\partial E},
\]
and we conclude that
\[
\limsup_{p \to 1^+} \capa_p(\Omega) \leq \abs{\partial E}
\]
for any bounded open set $E$ with smooth boundary containing $\Omega$.
In particular, considering a sequence of bounded sets $\Omega_k$ with smooth boundary containing $\Omega^*$ and with $\abs{\partial \Omega_k} \to \abs{\partial \Omega^*}$ as $k \to \infty$, provided in Lemma \ref{schmidt}, we get \eqref{chain-pcap-2}.

\medskip

\emph{Step 3.} The most involved step in the proof of \eqref{chain-pcap} is in the  inequality
\begin{equation}
\label{claim}
\capa_1(\Omega) \leq \liminf_{p \to 1^+} \capa_p(\Omega). 
\end{equation}
In order to show it, we treat separately the cases when $(M, g)$ satisfies an Euclidean-like Isoperimetric Inequality and the case where $(M, g)$ has nonnegative Ricci curvature and satisfies \eqref{minerbe}.

\smallskip 

\emph{Step 3 - Case 1: Assume that $(M, g)$ satisfies an Euclidean-like Isoperimetric Inequality.} 
Let $\mathrm{C}_{\rm Sob}$ be the  $L^1$-Sobolev constant, that is
\begin{equation}
\label{sobolev-inproof}
\left(\int_M {f}^{{n}/{(n-1)}}\dd\mu\right)^{(n-1)/n} \leq \mathrm{C}_{\rm Sob} \int_M \abs{\D f} \dd\mu
\end{equation}
for any nonnegative $f \in \mathscr{C}^{\infty}_0 (M)$. It is well known and easy to check that applying \eqref{sobolev-inproof} to $f^{\frac{n-1}{n-p}p}$ yields the $L^p$-Sobolev inequality
\begin{equation}
\label{psobolev-inproof}
\left(\int_M {f}^{p^*}\dd\mu\right)^{(n-1)/n} \leq \, \mathrm{C}_{n, p} \int_M \abs{\D f}^p \dd\mu
\end{equation}
with 
\begin{equation}
\label{cnp}
\mathrm{C}_{n, p} = \left[\frac{\mathrm{C}_{\rm Sob}(n-1)p}{(n-p)}\right],
\end{equation}
and
\begin{equation}
\label{pstar}
p^* \, = \, \frac{np}{n-p}
\end{equation}
for any $1 < p < n$. 
The argument yielding the key estimate \eqref{xu3} below is inspired by the proof of \cite[Theorem 3.2]{xu}.
By the definition of $1$-capacity and the H\"older inequality we get, for any $q > 0$,
\begin{equation}
\label{xu1}
\capa_1(\Omega) \leq \int\limits_{M} \abs{\D f^q} \dd \mu = q \int\limits_{M} f^{q-1} \abs{\D f} \dd\mu \leq q \left(\int_{M}f^{{(q-1)}\frac{p}{p-1}} \dd\mu\right)^{{(p-1)}/{p}} \left(\int_{M} \abs{\D f}^{p} \dd\mu\right)^{1/p}\!\!\!.
\end{equation}
Choose then
\begin{equation}
\label{q-xu}
q_p = 1 + p^*\frac{(p-1)}{p}
\end{equation}
and observe that $q_p > 1$. Then, applying the $L^p$-Sobolev inequality \eqref{psobolev-inproof} we obtain 
\begin{equation}
\label{xu2}
\capa_1(\Omega) \leq q_p \,{\rm C}_{n, p}^{\,(p-1)/p}\left(\int\limits_{M} \abs{\D f}^{p} \dd\mu\right)^{\!\!p^*(p-1)/p^2 + 1/p} \!\!\!\!\!\!\!\!\!\!\!\!\! = \,\,\, q_p \,\, {\mathrm{C}_{n, p}}^{\,(p-1)/p}\left(\int_{M} \abs{\D f}^{p} \dd\mu\right)^{\!\!(n-1)/(n-p)}.
\end{equation}
Taking the infimum in the rightmost hand side of the inequality above over any $f \in \mathscr{C}^{\infty}_c(M)$ satisfying $f \geq \chi_{\Omega}$ we are left with
\begin{equation}
\label{xu3}
\capa_1(\Omega) \leq  \,\,\, q_p \,\, {\mathrm{C}_{n, p}}^{\,(p-1)/p} \capa_p(\Omega).
\end{equation}
Letting $p \to 1^+$ in the above inequality, and observing that from the expressions  \eqref{cnp} and \eqref{q-xu} both $q_p$ and $\mathrm{C}_{n, p}^{(p-1)/p}$ tend to $1$, we get \eqref{claim} under the validity of an Euclidean-like Isoperimetric Inequality.

\smallskip

\emph{Step 3 - Case 2: Assume that $(M, g)$ has nonnegative Ricci curvature and satisfies \eqref{minerbe}.}
The superlinear volume growth following from \eqref{minerbe} guarantees that Holopainen's condition 
\begin{equation}
\label{holo-inproof}
\int_1^{+\infty} \frac{t}{\abs{B(O,  t)}} \dd t < +\infty
\end{equation}
is satisfied, and in particular \cite[Theorem 5.10]{holopainen1} ensures that $(M, g)$ is $p$-nonparabolic for $1<p< b$. 
Moreover, \cite[Theorem 3.8]{mari-rigoli-setti} gives the following decay estimate 
\begin{equation}
\label{li-yau-cond}
G_p(O, x) \leq \frac{{\mathrm{C}}^{1/(p-1)}}{(p-1)^2} \int_{d(O, x)}^\infty \left[\frac{t}{\abs{B(O, t)}}\right]^{1/(p-1)}\dd t.
\end{equation}
for any $x \in M \setminus B(O, R_1)$ for some $R_1 > 0$.  We claim that the same type of estimate holds for $u_p$.
Indeed, choosing $O \in \Omega$, a trivial comparison argument immediately yields
$u_p(x) \leq  G_p (O, x)/(\inf_{\partial \Omega} G_p) $. Moreover, it follows from the convergence results in \cite{mari-rigoli-setti} that the function $-(p-1) \log G_p(O, \cdot)$ converges locally uniformly to a continuous function (more precisely to a solution of the IMCF emanating from $O$) and in particular $-(p-1) \log G_p(O, \cdot)$ is uniformly bounded in the compact set $\partial \Omega$ uniformly in $p$ small enough. This implies that $(\inf_{\partial \Omega} G_p) \geq \mathrm{C}_2^{1/(p-1)}$ for some $\mathrm{C}_2$ not depending on $p$. Combining it with the above comparison we thus obtain  
\begin{equation}
\label{li-yau-cond-u}
u_p(x) \leq \frac{{\mathrm{C}_3}^{1/(p-1)}}{(p-1)^2} \int_{d(O, x)}^\infty \left[\frac{t}{\abs{B(O, t)}}\right]^{1/(p-1)}\dd t
\end{equation}
outside some ball.
The volume growth condition $\abs{B(O, t)} \geq C_{\rm vol} \,\, t^{b}$, following from \eqref{minerbe}, improves the above estimate to
\begin{equation}
\label{bound-u}
u_p(x) \leq \frac{1}{(b - p)} \frac{\mathrm{C}_4^{1/(p-1)}}{(p-1)} r(x)^{-(b-p)/(p-1)}, 
\end{equation}
where we denoted $r(x) = d(O, x)$, outside some big ball and for a positive constants $\mathrm{C}_4$ uniform in $p$ as $p \to 1^+$.

Extend now $u_p$ to be equal to $1$ on $\Omega$, for simplicity keeping the same name. We claim that ${{u}_p}^{q_p}$ is in $W^{1, 1}(M)$, for $p$ close enough to $1$, where $q_p$ is defined in \eqref{q-xu}. Clearly, \eqref{bound-u} implies that ${{u}_p}^{q_p} \in L^1(M)$ for $p$ close enough to $1$, and also $u_p$. Applying the H\"older inequality as in \eqref{xu1}  we also get
\begin{equation}
\label{xu-up}
\begin{split}
\int_M \abs{\D {u}_p^{q_p}} \dd\mu &\leq q_p \left(\int_{M}u_p^{p^*}\dd\mu\right)^{(p-1)/p}\left(\int_{M\setminus \overline{\Omega}} \abs{\D u_p}^p\dd\mu\right)^{1/p} \\ 
&=  q_p \left(\int_{M}u_p^{p^*}\dd\mu\right)^{(p-1)/p} \!\!\!\!\capa_p(\Omega)^{1/p} \\
&\leq q_p \left(\int_{M}u_p \dd\mu\right)^{(p-1)/p} \!\!\!\!\capa_p(\Omega)^{1/p}
\end{split},
\end{equation}
where the  equality is \eqref{p-cap-u-chap3}, and the last inequality is due to $u_p \leq u_p^{p*}$ following from $ 0 < u \leq 1$ and $p^* > 1$.
Since $u_p \in L^1(M)$ uniformly in $p$ close to $1$, and so does the $p$-capacity of $\Omega$ as $p \to 1^+$ by \eqref{chain-pcap-2} in \emph{Step 2}, \eqref{xu-up} implies that $u_p^{q_p} \in W^{1, 1} (M)$ uniformly as $p \to 1^+$.  
Since the class of competitors for $\capa_1(\Omega)$ can be easily relaxed to the class of $W^{1, 1}(M)$ without changing the infimum (see e.g. \cite[Chapter 2]{heinonen-book}), we can estimate $\capa_1(\Omega)$ from above with ${u}_p^{q_p}$ and get, from \eqref{xu-up}, that
\begin{equation}
\label{xu-up1}
\capa_1(\Omega) \leq q_p \left(\int_{M}u_p\dd\mu\right)^{(p-1)/p} \!\!\!\!\capa_p(\Omega)^{1/p}.
\end{equation}
We focus on the first round bracket in the rightmost hand side of \eqref{xu-up1}. Let us decompose $M$ in $\overline{B(O, R)} \cup( M \setminus \overline{B(O, R)})$. We obtain, with the aid of \eqref{bound-u},
\begin{equation}
\label{split-m}
\int_M {u}_p \dd\mu \leq \big\vert{\overline {B(O, R)}}\big\vert + \frac{1}{(b-p)} \frac{\mathrm{C}_4^{1/(p-1)}}{(p-1)} \int\limits_{M\setminus \overline{B(O, R)}}  r^{-(b-p)/(p-1)}  \dd \mu.
\end{equation}
We estimate the integral in the right hand side above as follows. We have
\begin{equation}
\label{series}
\begin{split}
\int\limits_{M \setminus \overline{B(O, R)}} r^{-(b-p)/(p-1)}  \dd \mu  &= 
\sum_{j=1}^{+\infty} \quad \int\limits_{B(O, 2 j R) \setminus \overline{B(O, j R)}} r^{-(b-p)/(p-1)}  \dd \mu \\ 
&\leq \sum_{j=1}^{+\infty}(jR)^{-[(b-p)/(p-1)]} \Big[\abs{B(O, 2jR)} - \abs{B(O, jR)}\Big] \\ 
&\leq \abs{\mathbb{B}^n} 2^n R^{n - [(b-p)/(p-1)]} \sum_{j=1}^{+\infty} j^{n - [(b-p)/(p-1)]} \leq \mathrm{C}_5  R^{n - [(b-p)/(p-1)]},
\end{split}
\end{equation}
where we used Bishop-Gromov to estimate $\abs{B(O, 2jR)} \leq \abs{\mathbb{B}^n}(2jR)^n$, and, in the last step, the convergence of the series, that holds true for $p$ close to $1$ with a value that is uniform in $p$. 
Resuming, we have, plugging the outcome of \eqref{series} into \eqref{split-m},
\begin{equation}
\label{stimafinale-u}
\int_M {u}_p \dd\mu \leq \Big\vert{\overline {B(O, R)}}\Big\vert +  R^n\frac{1}{(b-p){(p-1)}} \left[\frac{\mathrm{C}_6}{R^{(b-p)}}\right]^{1/(p-1)}
\end{equation}
for any $R$ big enough. Choose it in order to satisfy $R^{b-p} > \, \mathrm{C}_6$ for any $p$ close enough to $1$ to see that, with this choice, the second summand in the above right hand side vanishes in the limit as $p \to 1^+$. In particular, letting $p \to 1^+$ in \eqref{xu-up1} we infer
\begin{equation}
\capa_1(\Omega) \leq   \liminf_{p \to 1^+}\capa_p(\Omega)
\end{equation}
also under the assumption of nonnegative Ricci curvature with polynomial uniform volume growth \eqref{minerbe}.
\medskip

Arranging \eqref{chain-pcap-1}, \eqref{chain-pcap-2} and \eqref{claim} into \eqref{chain-pcap} completes the computation of the limit of $\capa_p(\Omega)$ as $p \to 1^+$.  
\end{proof}

In light of Theorem \ref{main-limit-2}, one could wonder whether the mere existence of a 
bounded strictly outward minimising hull  implies  a $p$-capacitary approximation of $\abs{\partial \Omega^*}$ in the sense of \eqref{limit-th}. The following counterexample, inspired in part by \cite[Section 4]{kotschwar}, provides  a manifold admitting an exhausting sequence of strictly outward minimising sets but such that $\capa_p(\Omega) = 0$ for any $\Omega \subset M$ with smooth boundary.
\begin{example}[A $p$-parabolic manifold with nonnegative Ricci curvature admitting $\Omega^*$]
\label{cigar}
Consider, for $n \geq 2$ the complete noncompact Riemannian manifold $(M, g)$ whose metric splits as $ g= \dd \rho \otimes \! \dd\rho + \tanh^2(\rho) g_{\Sf^{n-1}}$ on $[0, +\infty) \times \Sf^{n-1}$. For $n = 2$, this is the celebrated Hamilton's cigar soliton \cite{hamilton-cigar}. The Riemannian manifold $(M, g)$ has linear volume growth and nonnegative Ricci curvature. In particular, 
\begin{equation}
\label{holo-par}
\int_1^{+\infty} \frac{t}{\abs{B(O,  t)}} \dd t = +\infty,
\end{equation}
and by \cite[Proposition 1.7]{holopainen1} $(M, g)$ is $p$-parabolic for any $p > 1$. In particular, it is well known (see e.g. \cite[(1.5)]{holopainen1}) that $\capa_p(\Omega) = 0$ for any open bounded $\Omega \subset M$ with smooth boundary. However, the level sets of $\rho$ are strictly outward minimising, and so they provide the exhausting sequence (in fact a foliation) required in Theorem \ref{main-hull} for the well-posedness of $\Omega^*$. To check this last assertion, we invoke again the level set formulation of the Inverse Mean Curvature Flow, and the minimising properties of such level sets discovered in \cite{Hui_Ilm}. Let indeed $B = \{\rho < 1\}$, and consider on $M \setminus B$ the function 
\[
w = (n-1) \log \left(\frac{\tanh \rho}{\tanh 1}\right)
\] 
extended with continuity at $0$ on $B$. The function $w$ is immediately seen to satisfy the level set equation \eqref{imcf-eq} in $M \setminus \overline{B}$. Since the level sets of $w$ sweep out the whole $M$ as $\rho \to +\infty$, \cite[Smooth Flow Lemma 2.3]{Hui_Ilm} implies that $w$ is also a Weak Inverse Mean Curvature Flow in the sense recalled in Subsection \ref{sub-imcf}. Observe also that $w$ is proper, although $w\not\to + \infty$ as $\rho \to + \infty$ (compare with Remark \ref{properness}). In particular by \cite[Minimizing Hull Property 1.4, (ii)]{Hui_Ilm}, the bounded sets ${\rm Int}\{w \leq t\} = \{ w < t \}$ are strictly outward minimising for any $t \geq 0$. Theorem \ref{main-hull} applies and yields a well posed bounded strictly outward minimising hull to any bounded open set $\Omega \subset M$ with smooth boundary, whose area cannot be recovered by $p$-capacities, that vanish for any $p > 1$.  
\end{example} 

\section{The sharp Isoperimetric inequality in manifolds with nonnegative Ricci curvature}
\label{sec:isoperimetric}
As already described in the Introduction, even ensuring in a noncompact Riemannian manifold the existence of a mean-convex exhaustion does not seem to be a straightforward task.
One of the consequences of Theorem \ref{main-limit-1} becomes then providing such an exhaustion on manifolds with nonnegative Ricci curvature and Euclidean Volume Growth, that is in particular given by least area sets in presence of (bigger and bigger) smooth obstacles. This will be key in proving Theorem \ref{isoperimetric-theorem}, object of this section.

\smallskip

Let us now enlist some of the ingredients we are going to employ, in addition to the mean-convex exhaustion. The main one is the Willmore-type inequality  proved in \cite{Ago_Fog_Maz_1},
that we state for $\mathscr{C}^{1, 1}$-boundaries.
It can be deduced from its smooth version by approximating in the $W^{2, n-1}$-topology the functions locally describing the boundary involved by smooth ones.
We refer to Remark \ref{weak-mean} for the notion of weak mean curvature.
\begin{theorem}[Willmore-type inequality]
\label{willmore-th}
Let $(M, g)$ be a complete noncompact Riemannian manifold with nonnegative Ricci curvature and Euclidean volume growth. Let $\Omega \subset M$ be a bounded open set with $\mathscr{C}^{1, 1}$-boundary. Then, the inequality
\begin{equation}
\label{willmoref}
\int_{\partial \Omega} \left\vert\frac{\HH}{n-1}\right\vert^{n-1} \dd\sigma \geq \AVR \, \abs{\Sf^{n-1}}
\end{equation}
holds true, where with $\HH$ we are denoting the weak mean curvature of $\partial \Omega$.
\end{theorem}
In the following statement, we recall existence and regularity results of isoperimetric sets constrained in \emph{compact} subsets with smooth boundary. Given a complete Riemannian manifold $(M, g)$, an open bounded subset $U \subset M$ with smooth boundary and a value $0 < v < \abs{U}$, we say that a subset with finite perimeter $E_v \subset \overline{U}$ is isoperimetric of volume $v$ in $\overline{U}$ if $\abs{E_v} = v$ and
\begin{equation}
\label{isoperimetric-def}
P(E_v) \, = \, \inf\left\{P(F) \, \, \vert \,\, F\subset \overline{U}, \, \, \abs{F}= v\right\}.
\end{equation}
\begin{theorem}
\label{isoperimetric-aux}
Let $(M, g)$ be a complete Riemannian manifold of dimension $n$ with smooth boundary, and let $U \subset M$ be an open bounded subset with smooth boundary.  Then, for any $0 < v < \abs{U}$, there exists an isoperimetric set $E_v \subseteq \overline{U}$ of volume $v$. Moreover, $\partial E_v$ is of class $\mathscr{C}^{1, 1}$ in a neighbourhood of $\partial E_v \cap \partial U$, and a smooth constant mean curvature hypersurface in a neighbourhood of any point of $(\partial E_v \setminus \partial U) \setminus \mathrm{Sing}$, where $\mathrm{Sing} \subset (\partial E_v \setminus \partial U)$ is relatively closed and has Hausdorff dimension at most $n - 8$. If $2 \leq n \leq 7$, then $\mathrm{Sing}$ is empty.
\end{theorem}
\begin{proof}
The existence of the isoperimetric set $E_v$ is a standard application of the Direct Method.
The $\mathscr{C}^{1, 1}$-regularity of $E_v$ around $\partial U$ is the content of \cite[Theorem 6.15]{mondino-spadaro}, while the smoothness up to $\mathrm{Sing}$ with the above properties away from $\partial U$ is a consequence of the regularity properties of isoperimetric sets without obstacles, see e.g. \cite{morgan-regularity}. 
\end{proof}
\begin{remark}
\label{spadaro-remark}
As confirmed by \cite{spadaro-personal}, a modification in the proof of \cite[Theorem 4.3]{focardi-spadaro} allows to infer that, in the assumptions and notations above, if in addition $\partial U$ is mean-convex, then the boundary of $\partial E_v \cap \partial U$ decomposes in an $(n-2)$-dimensional regular part and in an $(n-2)$-rectifiable one. The modification is needed in order to take into account the volume constraint.
\end{remark}
In the following statement, we gather some fundamental relations satisfied by the mean curvatures of $\partial E_v$ and $\partial U$.  
\begin{proposition}
\label{mean-curviso-prop}
Let $(M, g)$ be a complete Riemannian manifold, and let $U \subset M$ be a bounded subset with smooth boundary. Then, for any $0 < v < \abs{U}$, the constrained isoperimetric set $E_v$ in $\overline{U}$ satisfies the following properties.
\begin{itemize}
\item[$(i)$] 
We have
\begin{equation}
\label{lowerboundH-claim}
\HH \geq \inf_{\partial U} \mathrm{H}_{\partial U},
\end{equation}
on $(n-1)$-almost any point in $\partial E_v \cap \partial U$.
We are denoting with $\HH_{\partial U}$ the mean curvature of $\partial U$, and with $\HH$ the weak mean curvature of $E_v$ defined $(n-1)$-almost everywhere on $\partial E_v \setminus \mathrm{Sing}$. The singular set $\mathrm{Sing}$ was defined in the statement of Theorem \ref{isoperimetric-aux}
\item[$(ii)$] We have
\begin{equation}
\label{upper-mean-iso}
\HH_v \geq \HH  
\end{equation}
at $(n-1)$-almost any point of $\partial E_v \setminus \mathrm{Sing}$
where $\HH_v$ is the constant value of the mean curvature on $(\partial E_v \setminus \partial U) \setminus \mathrm{Sing}$.
\end{itemize}
\end{proposition}
\begin{proof}
Arguing exactly as done for \cite[Lemma 6.10]{mondino-spadaro}, we get that the mean-curvature measure (see the aforementioned contribution for the definition) of $\partial E_v$, that since by Theorem \ref{isoperimetric-aux} $\partial E_v$ is $\mathscr{C}^{1, 1}$ at $\partial E_v \cap \partial U$, is represented by the $L^{\infty}$-function $\HH$, satisfies
\[
\int_K \HH \dd\sigma \geq \inf_{\partial U} \mathrm{H}_{\partial U} \int_{K} \dd\sigma
\]
for \emph{any} $K \subset \partial E_v \cap \partial U$ relatively compact. The arbitrariness of $K$ implies \eqref{lowerboundH-claim}.
To see \eqref{upper-mean-iso}, we can follow an argument outlined in the note of  Ritor\'e \cite{ritore-talk}, that we include for the reader's convenience. Let $\phi$ be a positive function compactly supported in $\partial E_v \setminus \mathrm{Sing}$, and $\psi$ be a positive function compactly supported in $(\partial E_v \setminus \partial U) \setminus \mathrm{Sing}$ , such that $\int_{\partial E_v} \phi = \int_{\partial E_v} \psi$. Let $\nu$ be the outward pointing unit normal to $\partial E_v$. Then, the variation of $\partial E_v$ given by $\psi \nu - \phi \nu$ fixes its volume and keeps it inside of $U$. Being $E_v$ isoperimetric, such variation is going to increase the area, and thus the first variation formula yields
\begin{equation}
\label{first-var}
\int\limits_{\partial E_v} \HH_v \phi \dd\sigma - \int\limits_{\partial E_v} \HH \psi \dd\sigma \geq 0.
\end{equation}  
The above inequality  implies that
\begin{equation}
\label{first-var-applied}
\HH_v \geq \frac{\int_{\partial E_v} \HH \psi \dd\sigma}{\int_{\partial E_v} \psi \dd\sigma},
\end{equation}
where we employed $\int_{\partial E_v} \phi = \int_{\partial E_v} \psi$. The arbitrariness of $\psi$ easily implies $\HH_v \geq \HH$ almost everywhere on $\partial E_v \setminus \mathrm{Sing}$, as claimed.  
\end{proof}
The above proposition together with the mean-convexity of our exhaustion and some geometric features of  manifolds with nonnegative Ricci curvature will enable us to 
apply the  Willmore-type inequality \eqref{willmoref} to constrained isoperimetric sets and result in a differential inequality satisfied by the isoperimetric profile of any of the exhausting sets, in turn yielding the Isoperimetric Inequality for these constrained isoperimetric sets, that will readily imply the same for any bounded set with smooth boundary $\Omega \subset M$. Let us see this in details, completing thus the proof of Theorem \ref{isoperimetric-theorem}.
\begin{proof}[Proof of Theorem \ref{isoperimetric-theorem}]
Let $\Omega \subset M$ be a bounded set with smooth boundary. We consider a bounded strictly outward minimising set $U$ satisfying $\Omega \Subset U$ with \emph{smooth} boundary. To produce it, we can argue as follows. Let $u$ be the capacitary potential of $\Omega$, that is, the solution to \eqref{pbp-chap3} with $p = 2$. Then, since such function vanishes at infinity, its superlevel sets exhaust $M$ and have compact closure. Being $u$ harmonic, it is smooth, and thus by Sard's theorem
we can immediately find $t \in (0, 1)$ such that $\{u = t\} = \partial \{u \geq t \}$ is smooth and $\Omega \Subset \{u \geq t\}$. Consider then the strictly outward minimising hull $\{u \geq t\}^*$ of $\{u \geq t \}$, existing by Theorem \ref{main-limit-1} and enjoying $\mathscr{C}^{1, 1}$-regularity by Theorem \ref{regularity-hulls}. This last sentence in particular holds true because the dimension is smaller or equal than $7$. Observe that the mean-convex boundary of $\{u \geq t\}^*$ cannot be minimal, because of the Willmore-type inequality of Theorem \ref{willmore-th}. Thus, we can approximate $\{u \geq t\}^*$ in $\mathscr{C}^1$ as in \cite[Lemma 5.6]{Hui_Ilm} to obtain a strictly outward minimising set $U$ with \emph{smooth} strictly mean-convex boundary and strictly enclosing $\Omega$, and thus obviously satisfying $\abs{\Omega} < \abs{U}$. 

\medskip

We consider now an isoperimetric set $E_v$ of $\overline{U}$ with volume $v < \abs{U}$. Such a set exists and enjoys the properties enlisted in Theorem \ref{isoperimetric-aux}. We first observe that  the boundary of  
$\partial E_v$ is strictly mean-convex and that the supremum of the mean-curvature is attained on $\partial E_v \setminus \partial U$.
If $\partial E_v \cap \partial U$ is of positive $(n-1)$-dimensional measure, this is a direct consequence of Proposition \ref{mean-curviso-prop} and of the strict mean-convexity of the smooth boundary $\partial U$.
If, on the other hand, $\partial E_v \cap \partial U$ is of null $(n-1)$-dimensional measure, then the $\mathscr{C}^{1, 1}$-boundary of $\partial E_v$ is of constant (weak) mean-curvature, and thus standard regularity theory implies that it is actually a smooth constant mean-curvature hypersurface. Thus, by Kasue's Theorem \cite[Theorem C]{kasue_minimal} (see also the  reformulation obtained in \cite[Theorem 1.7]{Ago_Fog_Maz_1}), its constant mean curvature must be strictly positive.

\medskip

Let us now apply the Willmore-type inequality \eqref{willmoref} to $E_v$. Set
 \begin{equation}
W = \inf\left\{ \,\int\limits_{\partial \Omega} \left\vert \frac{\HH}{n-1} \right\vert^{n-1} \, \, \, \Bigg\vert \,\,\, \Omega \text{ bounded with} \, \mathscr{C}^{1, 1} \, \text{boundary}\right\},
\end{equation}
that is strictly positive by means of \eqref{willmoref}.
Taking into account that the supremum of the mean curvature of $\partial E_v$ is $\HH_v$ and that $\partial E_v$ is mean-convex we get
\begin{equation}
\label{conto-will}
W \leq  \int_{\partial E_v} \left(\frac{\HH}{n-1}\right)^{n-1} \dd\sigma \leq \left(\frac{\HH_v}{n-1}\right)^{n-1} \abs{\partial E_v}.
\end{equation}
We are now going to link the above inequality satisfied by $\HH_v$ with the \emph{isoperimetric profile} of $\overline{U}$, that we recall to be defined as the function $I_{\overline{U}} : v \to \abs{\partial E_v}$. We recall that 
$I_{\overline{U}}$ is a continuous function, as proved in way greater generality in \cite{flores-nardulli}.
Let $\epsilon < 0$, and let $E^\epsilon \subset E_v$ be a subset of volume $\abs{E_v} + \epsilon$ obtained by deforming $\partial E_v$ with an inward deformation supported in $\partial E_v \setminus \partial U$. We have
\begin{equation}
\label{dini-estimate}
\frac{I_{\overline{U}}^{\frac{n}{(n-1)}}(v + \epsilon) - I_{\overline{U}}^{\frac{n}{(n-1)}}(v)}{\epsilon} \geq \frac{\abs{\partial E^{\epsilon}}^{\frac{n}{(n-1)}} - \abs{\partial E_v}^{\frac{n}{(n-1)}}}{\epsilon}. 
\end{equation}
Letting $\epsilon \to 0^-$ and coupling with \eqref{conto-will} we get
\begin{equation}
\label{dini1}
\begin{split}
\liminf_{\epsilon\to 0^{-}} \frac{I_{\overline{U}}^{\frac{n}{(n-1)}}(v + \epsilon) - I_{\overline{U}}^{\frac{n}{(n-1)}}(v)}{\epsilon} \geq \frac{n}{(n-1)} \, \abs{\partial E_v}^{\frac{1}{(n-1)}} \, \HH_v &\geq  n \left[\int_{\partial E_v} \left(\frac{\HH}{n-1}\right)^{n-1} \dd\sigma\right]^{\frac{1}{(n-1)}} \\
&\geq n W^{\frac{1}{(n-1)}}.
\end{split}
\end{equation}
The leftmost hand side of \eqref{dini1} constitutes the \emph{lower left Dini derivative} of $I_{\overline{U}}^{n/(n-1)}$ at $v$ that we are going to denote with $\D_{-} \big(I_{\overline{U}}^{n/(n-1)}\big) (v)$. Let us compare now with a reference conical isoperimetric profile 
\begin{equation}
\label{conical-profile}
I : v \to n^{\frac{(n-1)}{n}}W^{\frac{1}{n}} v^{\frac{n-1}{n}}. 
\end{equation} 
By \eqref{dini1} and the explicit  expression of $I$ we deduce
\begin{equation}
\label{monotonicity-profile}
\D_{-} \left(I_{\overline{U}}^{\frac{n}{(n-1)}} - I^{\frac{n}{(n-1)}}\right) \geq 0. 
\end{equation}
Since continuous functions with a nonnegative Dini derivative are known to be monotone nondecreasing (see e.g. \cite[Appendix I, Theorem 2.3]{liapunov}), and both $I_{\overline{U}}$ and $I$ are easily seen to vanish at zero, we get
\begin{equation}
\label{conto-profile2}
\frac{[I_{\overline{U}}(v)]^n}{v^{n-1}} \geq W n^{n-1}.
\end{equation}
Finally, recall that $\abs{\Omega} = v$, and since $\Omega \Subset U$ we have $I_{\overline{U}} (v) \leq \abs{\partial \Omega}$. Thus, \eqref{conto-profile2} and the Willmore-type inequality \eqref{willmoref} imply
\begin{equation}
\label{quasi-fine}
\frac{\abs{\partial \Omega}^n}{\abs{\Omega}^{n-1}} \geq \frac{[I_{\overline{U}}(v)]^n}{v^{n-1}} \geq W n^{n-1} \geq \AVR \abs{\Sf^{n-1}} n^{n-1} = \frac{\abs{\Sf^{n-1}}^n}{\abs{\mathbb{B}^n}^{n-1}} \AVR .
\end{equation}
We have in particular established the above inequality for any bounded set $\Omega \subset M$ with smooth boundary.
Since the leftmost hand side above approaches the rightmost hand side when computed on (smoothed out if necessary) geodesic balls $B(O, R)$ and let $R \to + \infty$, this completes the proof of \eqref{isoperimetricf}.

\medskip

Let us now prove the rigidity statement triggering when the isoperimetric inequality holds with the equality sign for a bounded $\Omega \subset M$ with smooth boundary. Consider as above $U\subset M$ a bounded set with smooth strictly outward minimising boundary such that $\Omega \Subset U$, and let $E_{v^*}$ be isoperimetric in $\overline{U}$ for the volume $v^* = \abs{\Omega}$. Then, \eqref{quasi-fine} implies that we can choose $\Omega = E_{v^*}$ and that \eqref{conto-profile2} holds with equality sign. Hence
\[
I_{\overline{U}}^{\frac{n}{(n-1)}}(v^*) - I^{\frac{n}{(n-1)}}(v^*) = 0. 
\]
Since, by \eqref{monotonicity-profile} such difference between isoperimetric profiles is monotone nondecreasing in $[0, v^*]$, and it is equal to zero at $v = 0$, we deduce that it is actually constant on the whole interval.
This implies that, for any $v \leq v^*$,
we have
\begin{equation}
\label{uguaglianza-tappeto}
\frac{\abs{\partial E_v}^{n}}{\abs{E_v}^{n-1}} = \AVR \frac{\abs{\Sf^{n-1}}^n}{\abs{\mathbb{B}^n}^{n-1}}.
\end{equation}
On the other hand, by the asymptotically optimal isoperimetric inequality in compact manifolds \cite[Théorème, Appendice C]{berard-meyer} applied in $\overline{U}$ to the subsets $E_v$, we have that, for any $\epsilon > 0$ small enough, there exists $v_\epsilon$ such that for any volume $v \leq v_\epsilon$ it holds
\begin{equation}
\label{almost-optimal-iso}
\frac{\abs{\partial E_v}^{n}}{\abs{E_v}^{n-1}} \geq (1 - \epsilon) \frac{\abs{\Sf^{n-1}}^n}{\abs{\mathbb{B}^n}^{n-1}}.
\end{equation}
Coupling it with \eqref{uguaglianza-tappeto}, this implies that $\AVR \geq (1 - \epsilon)$ for any $\epsilon > 0$ small enough, hence implying $\AVR = 1$ and thus, by Bishop-Gromov's rigidity, the isometry between $(M, g)$ and flat $\R^n$. Consequently, $\Omega$ is ball in $\R^n$ with respect to the flat metric.
\end{proof}

We now pass to deduce the Faber-Krahn inequality in manifolds with nonnegative Ricci curvature and Euclidean volume growth. We are dealing with the first eigenvalue $\lambda_1(\Omega)$ of $\Omega \subset M$, that is the smallest $\lambda > 0$ such that 
\begin{equation}
\label{eigenvalue-prob}
\begin{cases}
\,\,-\Delta{u}= \lambda u & \mbox{in} \,\, \Omega 
\\
\,\,\,\,\,\,\,u=0 & \mbox{on}\,\, \partial\Omega 
\end{cases}
\end{equation}
admits a solution.
The complete statement reads as follows.
\begin{theorem}[Faber-Krahn inequality on manifolds with nonnegative Ricci]
\label{faber-th}
Let $(M, g)$ be a complete noncompact Riemannian manifold with nonnegative Ricci curvature and Euclidean volume growth. Let $\Omega \subset M$ be a bounded subset with smooth boundary of volume $v$. Then, we have
\begin{equation}
\label{faberf}
\lambda_1(\Omega) \geq \AVR^{2/n} \lambda_1(\mathbb{B}_v),
\end{equation}
where $\mathbb{B}_v$ is any metric ball of flat $\R^n$ with volume $v$. In the above formula $\lambda_1(\Omega)$ and $\lambda_1(\mathbb{B}_v)$ denote the first eigenvalue of $\Omega$ and $\mathbb{B}_v$ with respect to $g$ and the flat metric of $\R^n$ respectively.
Moreover, equality holds in \eqref{faberf} if and only if $(M, g)$ is isometric to flat $\R^n$ and $\Omega = \mathbb{B}_v$.
\end{theorem}
In order to show it, we rely on the following, well known general principle, named after P\'{o}lya and Szeg\"o.
\begin{proposition}[P\'{o}lya-Szeg\"o Principle]
\label{polia-szego}
Let $(M, g)$ be a Riemannian manifold, and assume that for any bounded $\Omega \subset M$ with smooth boundary there holds
\begin{equation}
\label{generalised-iso}
\frac{\abs{\partial \Omega}^n}{\abs{\Omega}^{n-1}} \, \geq \, \frac{\abs{\Sf^{n-1}}^n}{\abs{\mathbb{B}^n}^{n-1}} \mathrm{C}_g 
\end{equation}
for some $\mathrm{C}_{g}$ independent of $\Omega$. Let, for $\Omega \subset M$ a bounded set with smooth boundary of volume $v$, $f \in \mathscr{C}^{\infty}(\overline{\Omega})$ be a positive function with $f = 0$ on $\partial \Omega$. Then, there exists a function $F\in W^{1, 2}_0(\mathbb{B}_v)$, where $\mathbb{B}_v \subset \R^n$ is a ball with respect to the flat metric of volume $v$ such that
\begin{equation}
\label{polia-szegof}
\int_{\Omega} \abs{\D f}^2 \dd\mu \geq \mathrm{C}_{g}^{\frac{2}{n}} \int_{\mathbb{B}_v} \abs{\D F}^2_{\R^n} \dd\mu_{\R^n} \qquad \qquad \int_\Omega f^2 \dd\mu = \int_{\mathbb{B}_v} F^2 \dd \mu_{\R^n}.
\end{equation} 
Moreover, equality holds in the inequality above if and only $\Omega$ satisfies equality in \eqref{generalised-iso}.
\end{proposition} 
We are reporting for completeness a proof of the above Proposition in Appendix B. We can now easily deduce Theorem \ref{faber-th}.
\begin{proof}[Proof of Theorem \ref{faber-th}]
Just recalling that the first Dirichlet eigenvalue minimises the Rayleigh quotient, we get, applying \eqref{polia-szegof} to the solution of \eqref{eigenvalue-prob} with $\lambda = \lambda_1(\Omega)$, and observing that by the sharp Isoperimetric inequality \eqref{isoperimetric-inequality-intro} we have $\AVR = \mathrm{C}_g$, we get
\begin{equation}
\label{rayleigh}
\lambda_1(\Omega) = \frac{\int_\Omega \abs{\D f}^2 \dd\mu}{\int_\Omega f^2 \dd\mu} \geq \AVR^{2/n} \frac{\int_{\mathbb{B}_v} \abs{\D F}^2_{\R^n} \dd\mu_{\R^n}}{\int_{\mathbb{B}_v} F^2 \dd\mu_{\R^n}} \geq \AVR^{2/n} \lambda_1(\mathbb{B}_v),
\end{equation}
that is \eqref{faberf}. If equality holds in \eqref{faberf}, then, equalities hold in \eqref{rayleigh}, that implies equality holds in the inequality of \eqref{polia-szegof}. The rigidity statement of Proposition \ref{polia-szego} then shows that $\Omega$ satisfies equality in the Isoperimetric Inequality of $(M, g)$, and then the rigidity statement of Theorem \ref{isoperimetric-theorem} applies, yielding the isometry of $(M, g)$ with $(\R^n, g_{\R^n})$ and the isometry of $\Omega$ with $\mathbb{B}_v$. 
\end{proof}
\setcounter{equation}{0}
\renewcommand\theequation{A.\arabic{equation}}
\section*{Appendix A}

We furnish a proof of  Theorem \ref{existence-potential}. Despite seemingly well known, it is not easy to find in literature a complete and in particular self-contained proof of this existence and uniqueness result. Anyway, no original arguments appear below, and we are referring mainly to \cite{holopainen-quasiregular}, \cite{heinonen-book} and \cite{pigola-setti-troyanov}. 
\begin{proof}[Proof of Theorem \ref{existence-potential}]
Let $B(O, R)$ be a geodesic ball containing $\Omega$.  Let $\psi \in \mathscr{C}^{\infty}_c (B(O, R))$ satisfy $\psi = 1$ on $\Omega$.
Then, we can find a solution to the Dirichlet problem for the $p$-Laplacian with boundary values $1$ on $\partial \Omega$ and $0$ on $\partial B(O, R)$, namely a weakly $p$-harmonic function $u_R$ on $B(O, R) \setminus \Omega$ satisfying $u - \psi \in W^{1, p}_0 (B(O, R) \setminus \overline{\Omega})$. It is immediately deduced from Tolksdorf's Comparison Principle for $p$-harmonic functions (see for example \cite[Lemma 3.18]{heinonen-book} that if $\tilde{u}$ is another such function relative to another $\tilde{\psi} \in W^{1, p}_0 (B(O, R))$ then $u = \tilde{u}$ on $B(O, R) \setminus \overline{\Omega}$.
Defining the $p$-capacity of $\Omega$ relative to $B(O, R)$ as 
\begin{equation}
\label{p-cap-rel}
\capa_p(\Omega, B(O, R)) = \inf\left\{\int_{M} \abs{\D f}^p \dd\mu \, \, \Big\vert \, \, f \geq \chi_\Omega, \, f \in \mathscr{C}^{\infty}_c (B (O, R))\right\},
\end{equation} 
we immediately see through a very standard argument that $u_R$ can be approximated in $W^{1, p}$ by a sequence of admissible competitors in \eqref{p-cap-rel}, and thus
\begin{equation}
\label{inequality-pcap-rel}
\capa_p(\Omega, B(O, R)) \leq \int\limits_{B(O, R) \setminus \overline{\Omega}} \abs{\D u_R}^p \dd\mu.
\end{equation}
On the other hand, observing that in the definition of weakly $p$-harmonic functions on $B(O, R) \setminus \overline{\Omega}$ the class of tests can easily be relaxed to $W^{1, p}_0 (B(O, R) \setminus \overline{\Omega})$, we get
\begin{equation}
\label{passaggio-pcap-rel}
\int\limits_{B(O, R) \setminus \overline{\Omega}} \abs{\D u_R}^p \dd\mu = \int\limits_{B(O, R) \setminus \overline{\Omega}} \left\langle \abs{\D u_R}^{p-2} \D u_R \, \vert \, \D u_R\right\rangle \dd\mu = \int\limits_{B(O, R) \setminus \overline{\Omega}} \left\langle \abs{\D u_R}^{p-2} \D u_R \, \vert \, \D \psi\right\rangle \dd\mu,
\end{equation} 
where the last equality is deduced from the definition of weak $p$-harmonicity with $u - \psi \in W^{1, p}_0 (B(O, R) \setminus \overline{\Omega}) $ as test function. Applying the H\"older inequality to the right hand side of \eqref{passaggio-pcap-rel}, we are immediately left with
\begin{equation}
\label{holder-pcap-rel}
\int\limits_{B(O, R) \setminus \overline{\Omega}} \abs{\D u_R}^p \dd\mu \, \, \leq \int\limits_{B(O, R) \setminus \overline{\Omega}} \abs{\D \psi}^p \dd\mu
\end{equation}
for any $\psi \in \mathscr{C}^{\infty}_c (B (O, R))$ satisfying $\psi = 1$ on $\Omega$. Since, clearly, a sequence of functions satisfying the same assumptions on $\psi$ suffices to realise the relative $p$-capacity of $\Omega$ (see for example \cite[(2.2.1)]{mazia} for details), we get passing to the limit through such a sequence in \eqref{holder-pcap-rel} that we can also take the opposite inequality sign in \eqref{inequality-pcap-rel}, obtaining
\begin{equation}
\label{pcap-rel-identity}
\capa_p(\Omega, B(O, R)) = \int\limits_{B(O, R) \setminus \overline{\Omega}} \abs{\D u_R}^p \dd\mu.
\end{equation}
Let us finally show that passing to the limit as $R \to +\infty$ yields a solution to \eqref{pbp-chap3}. In what follows, we are frequently extending $u_R$ to $0$ outside $B(O, R)$ without explicitly mentioning it. The deep $\mathscr{C}^{1, \beta}_{\mathrm{loc}}$-estimates for $p$-harmonic functions give, for any compact set $K$ of $M \setminus \Omega$, a constant $\mathrm{C}$ so that
\begin{equation}
\label{holder-estimate}
\abs{u_R (x) - u_R(y)} \leq \mathrm{C}  \, d(x, y)^\beta
\end{equation} 
and
\begin{equation}
\label{holdergrad-estimate}
\Big\vert \abs{\D u_R}(x) - \abs{\D u_R}(y) \Big\vert \leq \mathrm{C} \,  d (x, y)^\beta.
\end{equation}
The constant $\mathrm{C}$ depends on the dimension, on $p$, on smooth quantities related to the underlying metric $g$ on $K$ and on the $W^{1, p}$ norm of $u_R$ on $K$. By the Maximum Principle for $p$-harmonic functions, $0 \leq u_R \leq 1$ on $K$. Moreover, by \eqref{pcap-rel-identity} and the obvious monotonicity of the $\capa_p(\Omega, B(O, R))$ as $R$ increases, 
we have
\begin{equation}
\label{uniformlyboundedpcap}
\int\limits_K \abs{\D u_R}^p \dd\mu \leq \int\limits_{B(O, R_0)} \abs{\D u_{R_0}}^p \dd\mu
\end{equation}
for any $R_0 > 0$ such that $K \Subset B(O, R_0)$ and any $R > R_0$. In particular, the constant $\mathrm{C}$ is uniform in $R$, and by Arzel\`a-Ascoli applied to the sequence of $u_R$ and the sequence of the gradients we deduce that, up to a subsequence, the sequence $u_R$ converges in $\mathscr{C}^{1}_{\mathrm{loc}}$ to a $\mathscr{C}_{\mathrm{loc}}^1$ function $u$. Moreover, as it is immediately seen from the weak formulation of $p$-harmonicity, such $u$ is weakly $p$-harmonic on $M \setminus \overline{\Omega}$. Since $\partial \Omega$ is in particular Lipschitz, it satisfies an interior and exterior cone condition. In particular, by \cite{beirao-conti} (see also \cite{beirao-nonlinear}) $u_R$ attains continuously the datum on $\partial \Omega$. Thus, \cite[Theorem 1]{liebermann} implies that the $\mathscr{C}^{1, \beta}_{\mathrm{loc}}$-estimates for $u_R$  are valid up to the boundary, that is $K$ was allowed to contain portions of $\partial \Omega$, and thus the convergence of the $u_R$ takes place also on $\partial \Omega$, where then the sequence obviously converges to $1$. This argument also shows that $u$ takes the initial datum on $\partial \Omega$ in $\mathscr{C}^{1, \beta}$. 
Moreover, since by Tolksdorf's Hopf Lemma for $p$-harmonic functions \cite{tolksdorf} we have $\abs{\D u} > 0$ on $\partial \Omega$, the continuity up to the boundary of the gradient ensures that $\abs{\D u} > 0$ in a neighbourhood of $\partial \Omega$, where thus the solution is a smooth classical solution by quasilinear elliptic regularity, and in particular  it attains the datum in $\mathscr{C}^{k, \beta}$ for some $\beta \in (0, 1)$ if $\partial \Omega$ belongs to $\mathscr{C}^{k, \alpha}$. 

To complete the existence part of the proof, we just have to check that $u$ vanishes at infinity. This is immediately deduced from the vanishing at infinity of $G_p$. Indeed, since $G_p$ is positive, a straightforward comparison argument for the approximators $u_R$ yields a constant $\mathrm{C}$ independent of $R$ so that $0 < u_R \leq \mathrm{C} \, G_p(O, \cdot)$,
that by the convergence shown above, implies, sending $R \to +\infty$, 
\begin{equation}
\label{basic-barrier}
0 < u \leq \mathrm{C} \, G_p(O, \cdot).
\end{equation}
This last estimate clearly implies the vanishing at infinity of $u$.
\smallskip

Briefly, we show uniqueness. Let $v$ be any other solution to \eqref{pbp-chap3}, and let $k \in \N$. Let $R_k$ be such that $u \leq v + 1/k$ on (a smoothed out approximation of) $\partial B(O,{R_k})$. Such a radius surely exists by the vanishing at infinity of $u$. Then, since $u$ and $v$ achieve the same value on $\partial \Omega$, the Comparison Principle for $p$-harmonic functions applied to $u$ and $v + 1/k$ on $B(O, R_k) \setminus \overline{\Omega}$ shows that $u \leq v + 1/k$ on this set. Letting $k \to +\infty$ we obtain $u \leq v$ on $M \setminus \Omega$. Exchanging the roles of $u$ and $v$ gives the opposite inequality, showing that $u = v$, that is uniqueness.
\smallskip

Now, we check that $u$ realises the $p$-capacity of $\Omega$. This again will come from the properties of the approximators. 
Since $u_R \to u$ as $R\to \infty$ pointwise, and $\int_K \abs{\D u_R}^p \dd\mu$ is uniformly bounded in $R$ for any compact $K \Subset M\setminus \Omega$ by \eqref{uniformlyboundedpcap},
we can invoke the basic but very useful \cite[Lemma 1.33]{heinonen-book} to infer that $u \in L^p (M \setminus \overline{\Omega})$ and that $\D u_R$ converges weakly to $\D u$. In particular, by lower semicontinuity of the norm we
have
\[
\int\limits_{M \setminus \overline{\Omega}} \abs{\D u}^p \dd\mu \leq \liminf_{R \to \infty} \int\limits_{M \setminus \overline{\Omega}} \abs{\D u_R}^p \dd\mu,
\]
and since the righthand side is uniformly bounded again by \eqref{uniformlyboundedpcap}, we conclude that $u \in W^{1, p}(M \setminus \overline{\Omega})$. Since it vanishes at infinity, it actually belongs to $W_0^{1, p}(M \setminus \overline{\Omega})$. We can thus argue exactly as done for \eqref{pcap-rel-identity} to show that $u$ realises the $p$-capacity of $\Omega$, that is \eqref{p-cap-u-chap3}.
\end{proof}

\setcounter{equation}{0}
\setcounter{theorem}{0}
\renewcommand{\thetheorem}{B.\arabic{theorem}}
\renewcommand\theequation{B.\arabic{equation}}
\section*{Appendix B}
We furnish a proof of Proposition \ref{polia-szego}. Although this is very classical, and we are substantially reporting part of the  one given for \cite[Chapter IV, Theorem 2]{chavel-eigenvalues}, we were not able to find in literature an explicit dependence on the isoperimetric constant and a consequent discussion of the extremal case.
\begin{proof}[Proof of Proposition \ref{polia-szego}]
Let  $V: [0, T] \to \R$, where $T = \max_{\overline{\Omega}} f$, the function defined by 
$
V(t) = \abs{\{f \geq t\}}.
$
This function is easily seen to be continuous and actually $\mathscr{C}^{\infty}$ on regular values of $f$. Consider $B_t = \mathbb{B}_{\abs{\{f \geq t\}}}$, that is, the ball in $\R^n$ with (Euclidean) volume equal to $V(t)$. This defines a bijective function $\rho: [0, T] \to [0, \rho(0)]$ such that
$
V(t) = \abs{B(\rho(t))}_{\R^n}$,
where $B(\rho(t))$ is any ball of radius $\rho(t)$. 
Define finally $F: \overline{\mathbb{B}_v} \to \R$ by $F = \rho^{-1} \circ r$, where $r$ is the distance from the center of $\mathbb{B}_v$. It is shown exactly as in \cite{chavel-eigenvalues} that with this definition the identity in \eqref{polia-szegof} is satisfied. By coarea formula, we have
\begin{equation}
\label{passaggio-faber}
\int_{\Omega} \abs{\D f}^2 \dd\mu = \int\limits_0^T \int\limits_{\{f = t\}} \abs{\D f} \dd\sigma \dd t \geq \int\limits_0^T\abs{\{f = t\}}^2 \int\limits_{\{f = t\}} \frac{1}{\abs{\D f}} \dd\sigma \dd t
\end{equation}
On the other hand, by  \eqref{generalised-iso}, 
\begin{equation}
\label{application-iso}
\abs{\{f = t\}}^2  \, \geq \, \abs{\{f \geq t\}}^{2\frac{n-1}{n}} \abs{\Sf^{n-1}}^2 \, \abs{\mathbb{B}^n}^{2\frac{n-1}{n}} \mathrm{C}_{g}^{\frac{2}{n}} \, = \, \abs{B_t}_{{\R^n}} ^{2\frac{n-1}{n}} \, \abs{\Sf^{n-1}}^2 \, \abs{\mathbb{B}^n}^{2\frac{n-1}{n}} \mathrm{C}_{g}^{\frac{2}{n}} \, = \, \abs{\partial B_t}_{{\R^n}}^2  \mathrm{C}_{g}^{{2}/{n}},
\end{equation}
where in the last equality we used the fact that balls in $\R^n$ satisfy equality in the Isoperimetric Inequality.
Using 
that
\[
\rho'(t)= - \frac{1}{\abs{\partial B_t}_{\R^n}} \int\limits_{\{f = t\}} \frac{1}{\abs{\D f}} \dd\sigma \qquad \quad \left(\rho^{-1}\right)'(\rho(t)) = \frac{1}{\rho'(t)},
\]
we get,
plugging the outcome of \eqref{application-iso} in \eqref{passaggio-faber} and using again the coarea formula,
\begin{equation}
\label{chain-poliaszego}
\int_{\Omega} \abs{\D f}^2 \dd\mu \,\geq\, \mathrm{C}_{g}^{{2}/{n}} \int\limits_0^T \left[\left(\rho^{-1}\right)'\right]^2 \abs{\partial B_t}_{\R^n} \rho'(t) \dd t \, = \, \mathrm{C}_{g}^{{2}/{n}}\int_{\mathbb{B}_v} \abs{\D F}^2_{\R^n} \dd\mu_{\R^n},
\end{equation}
as desired.

Assume now that equality holds in the inequality of \eqref{polia-szegof}. Then, we deduce from \eqref{chain-poliaszego} and \eqref{application-iso} that 
\begin{equation}
\label{integral-poliaszego}
\int\limits_0^T\abs{\{f = t\}}^2  - \abs{\partial B_t}_{{\R^n}}^2  \mathrm{C}_{g}^{\frac{2}{n}} \dd t = 0,
\end{equation}
that, by the nonnegativity of the integrand, implies that the sets $\{f \geq t\}$ satisfy equality in the Isoperimetric Inequality of $(M, g)$ for any regular value $t$. Letting $t \to 0^+$, that is possible by Sard's Theorem, we conclude that $\Omega$ satisfy equality in \eqref{generalised-iso}.
\end{proof}

\end{document}